\documentclass[11pt]{amsart}

\usepackage{amsmath}
\usepackage{amsfonts}
\usepackage{amssymb}

 \advance\hoffset by -0.0 truecm

\newtheorem{theorem}{Theorem}[section]

\newtheorem{lemma}[subsection]{Lemma}
\newtheorem{proposition}[subsection]{Proposition}
\theoremstyle{definition}
\newtheorem{definition}[subsection]{Definition}
\theoremstyle{remark}
\newtheorem{remark}[theorem]{Remark}
\numberwithin{equation}{section}

\DeclareMathOperator{\re}{\hbox{Re}}
\DeclareMathOperator{\im}{\hbox{Im}}

\DeclareMathOperator{\sym}{sym}

\begin{document}
\title[Semi-Riemannian geometry with constraint]
{Semi - Riemannian  Geometry with Nonholonomic Constraints}

\author{Anna Korolko,\ Irina Markina}

\address{Department of Mathematics,
University of Bergen, Johannes Brunsgate 12, Bergen 5008, Norway}

\email{anna.korolko@uib.no}

\address{Department of Mathematics,
University of Bergen, Johannes Brunsgate 12, Bergen 5008, Norway}

\email{irina.markina@uib.no}

\thanks{The authors are partially supported by the grant of the Norwegian Research Council \# 177355/V30, by the grant of the European Science Foundation Networking Programme HCAA, and by the NordForsk Research Network Programme \# 080151}




\begin{abstract}

In the present article the geometry of semi-Riemannian manifolds
with nonholonomic constraints is studied. These manifolds can be
considered as analogues to the sub-Riemannian manifolds, where the
positively definite metric is substituted by a nondegenerate
metric. To study properties of the exponential map the Christoffel
symbols and other differential operators were introduced. We study
solutions of the Hamiltonian system and their projections into the
underlying manifold. The explicit formulae were found for a
specific example of a semi-Riemannian manifold with nonholonomic
constraints.
\end{abstract}

\maketitle

\section{Introduction}\label{sec:1}

Sub-Riemannian manifolds and the geometry introduced by bracket generating distributions of smoothly varying $m$-dimensional planes is widely studied interesting subject, which has applications in control theory, quantum physics, C-R geometry, the theory of principal bundles, and other areas. The main difference of the sub-Riemannian manifold from a Riemannian one is the presence of a smooth subbundle of the tangent bundle, generating the entire tangent bundle by means of the commutators of vector fields. The subbundle, which  is often called horizontal, is equipped with a positively definite metric that leads to the triple: manifold, horizontal subbundle, and Riemannian metric on the horizontal subbundle, which is called a {\it sub-Riemannian manifold}. The foundation of the sub-Riemannian geometry can be found in~\cite{LiuSussmann,Montgomery2,Montgomery,Strichartz}. The following question can be asked. What kind of geometrical features will have the mentioned triplet if we change the positively definite metric on the subbundle to an indefinite nondegenerate metric? We use the term {\it semi-Riemannian} to emphasis that the considered metric is nondegenerate in contrast to the positively definite metric, that referred as Riemannian. As it is known to the authors the present work is the first attempt to study systematically the geometry of semi-Riemannian manifolds with nonholonomic constraints, that we called the {\it sub-semi-Riemannian manifolds} or shortly {\it ss-manifolds}. In the present paper we study the exponential map and solutions to the Hamiltonian system that has no established terminology in the literature and can be called geodesics or extremals, see, for instance~\cite{LiuSussmann,Strichartz}. The ss-manifolds have their own peculiarity that distinguishes them from the sub-Riemannian and semi-Riemannian manifolds.
The simplest example of a semi-Riemannian manifold with nonholonomic constraints is the Heisenberg group equipped with the Lorentzian metric and it has been considered in~\cite{Groch1,Groch2,KM}. It was shown in~\cite{KM} that in contrast with the Heisenberg group with positively definite metric the Lorentzian type of the Heisenberg group possesses the uniqueness of extremals both of timelike or spacelike type. The structure of the article is the following. Section~\ref{sec:2} is devoted to main definitions. The collection of technical lemmas concerning Christoffel symbols is proved in Section~\ref{sec:3}. In Section~\ref{sec:4} the extremals and exponential map are introduced, the extendability of extremals and Gauss lemma are shown. Some properties of the length are also studied. Section~\ref{sec:5} is devoted to the differential properties of the exponential map. It is shown that the exponential map possesses an analogue of ``local diffeomorphism'' property, although is it not a diffeomorphism at the origin. The last Section~\ref{sec:6} consists of the example of ss-manifold, where the explicit formulae of extremals are found.

\section{Main definitions}\label{sec:2}

Let $M$ be a connected $n$-dimensional, $n\geq 3$, $C^{\infty}$-manifold. Let $T_x$ and $T^*_x$ denote the tangent
and cotangent spaces at a point $x\in M$, and $\langle Y,\xi\rangle$ the
pairing between them, $Y\in T_x$, $\xi\in T^*_x$. The tangent and cotangent bundles are denoted by $T$ and $T^*$ respectively. Fix an integer
$m$, $1<m<n$. Let $S$ be a fixed subbundle of the tangent
bundle $T$, $S=\bigcup\limits_{x\in M}S_x$, $S_x$ be a fiber over $x$, of the rank $m$. A subbundle $S$ will be
called {\it bracket generating} or complete {\it nonholonomic}, if the vector fields which are
sections of $S$, together with all brackets span $T_x$ at each
$x\in M$. In this case any two points in $M$ can be connected by a piecewise smooth curve $\gamma(s)$ such that the tangent vector $\dot\gamma(s)$ belongs to $S_{\gamma(s)}$ at each point $\gamma(s)$ where the tangent vector exists. The bracket generating subbundle $S$ is called the {\it horizontal bundle} or {\it horizontal distribution} and a curve $\gamma(s)$ satisfying $\dot\gamma(s)\in S_{\gamma(s)}$ is called the {\it horizontal curve}. A result of Chow~\cite{Chow}, see also~\cite{Rashevsky}, guarantees the connectivity of $M$ by a horizontal curve. The necessary and sufficient condition on connectivity by curves tangent to a given distribution of a smooth manifold can be found in~\cite{Sussmann}. We notice that the connectivity of a manifold by horizontal curves tangent to a given distribution depends only on properties of the distribution and not on any metric defined on it or on the tangent bundle.  If $Y\in S$, let $S+[Y,S]$ denote the subbundle
of $T$ spanned by $S$ and all the vector fields
$[Y,X]$, where $X$ varies over
sections $S$. A fiber at a point $x\in M$ is written as $S_x+[Y(x),S_x]\in T_x$ with $Y(x)\in S_x$. Similarly we define  bracket$(k,Y)$ inductively by
bracket$(2,Y)=S+[Y,S]$ and
bracket$(k,Y)=S+[\mbox{bracket}(k-1,Y),S]$. More generally we set
bracket$(2,S)=S+[S,S]$ and
bracket$(k,S)=S+[\mbox{bracket}(k-1,S),S]$. A restriction of a bundle to $x\in M$ is denoted by writing the subscript $x$, for example: bracket$(k,Y(x))=S_x+[\mbox{bracket}(k-1,Y(x)),S_x]\in T_x$. We will say that a vector field $Y\in S$ is a $k$-step generator if bracket$(k,Y(x))=T_x$ for all $x\in M$.
Similarly, $S$ will be said to be
$k$-step bracket generating distribution if bracket$(k,S_x)=T_x$ for
every $x$. From now on we work with a distribution $S$ possessing the bracket generating property.

By analogy with the sub-Riemannian metric we give the following definition.
\begin{definition} Let $M$ be a smooth manifold, $S$ be a bracket generating subbundle of the tangent bundle $T$. A {\it sub-semi-Riemannian metric} $Q$ on $S$ is
a smoothly varying in $x$ nondegenerate quadratic form $Q_x$
on $S_x$. We abbreviate the long and tedious name of sub-semi-Riemannian metric by the term {\it ss-metric}. We call the pair $(S,Q)$ the sub-semi-Riemannian (ss-Riemannian) structure on $M$.
\end{definition}

We remind that the index $\nu$ of a metric is the maximal dimension of the space $V_x\subset S_x$, where the form $Q_x$ is negatively definite. If $\nu=1$ then we call the ss-metric the {\it sub-Lorentzian} metric following the tradition in semi-Riemannian geometry. The ss-metric with the index $\nu=0$ is just a sub-Riemannian metric. Given $Q_x$, we may define a linear mapping $g_x\colon
T^*_x\to T_x$ as follows: for given $\xi\in T^*_x$, the linear mapping
$W\to \langle W,\xi\rangle$, $W\in S_x$ can be represented uniquely as $W\to
Q_x(W,X)$ for some $X\in S_x$, then $X$ is chosen to be $g_x\xi$. The map $g_x$ is called a cometric and completely defined by the two following relations:
\begin{itemize}
\item[(i)]{image of $T^*_x$ under $g_x$ is $S_x$,}
\item[(ii)]{$g_x$ and $Q_x$ are related by the identity\begin{equation}\label{eq:2.1}
   Q_x(W,g_x\xi)=\langle W,\xi\rangle \quad \mbox{for all}\;\; W\in S_x.
\end{equation}}
\end{itemize}

\begin{lemma}
If $Q_x$ is symmetric, nondegenerate and has index $\nu$, then $g_x$ is symmetric, nondegenerate and has index $\nu$.
\end{lemma}
\begin{proof}
 We understand the action of the cometric $g$ on $T^*\times T^*\to\mathbb R$ (omitting $x$)  as following:
$g(\xi,\psi)=\langle g\xi,\psi\rangle$ for any two covectors $\xi$ and $\psi$ from $T^*$.

Thus by definition of the cometric $g$ we have $g(\psi,\xi)=\langle g\psi,\xi\rangle=Q(g\psi,g\xi)$, which equals to $Q(g\xi,g\psi)=\langle g\xi,\psi\rangle=g(\xi,\psi)$ by the symmetry of the ss-metric $Q$.

Now, having the nondegeneracy of $Q$ we prove the nondegeneracy of $g$, that is if $g(\xi,\psi)=0$ for any $\psi\in T^*$ then $\xi\equiv 0$. First of all, we notice that the pairing $\langle Y,\xi\rangle$ is not degenerate by~\eqref{eq:2.1}. Then, taking arbitrary $\psi\in T^*$ and setting
$Y=g\psi$, we obtain $$g(\psi,\xi)=\langle Y,\xi\rangle=0\quad\text{ for any}\quad Y\in S.$$ This implies that $\xi\equiv 0$ by the nondegeneracy of pairing and we conclude that $g$ is nondegenerate cometric.

Let $S^{\bot}_x$
denote the kernel of $g_x$, and $S^{\bot}\subseteq T^*$ be the subbundle with fibers
$S^{\bot}_x$. Then $g_x\colon T^*_x/S^{\bot}_x\to S_x$ is bijection. The relation~\eqref{eq:2.1} shows that the index of $Q_x$ and $g_x$ coincides for any $x\in M$ and that $g_x$ degenerates on $S^{\bot}_x$.
\end{proof}

Conversely, given a symmetric nondegenerate linear operator
$g_x\colon T^*_x\to T_x$ with image $S_x$, there is a unique
nondegenerate quadratic form $Q_x$ satisfying \eqref{eq:2.1}. We
write $g^{jk}_x$, $j,k=1,\ldots n$ for the symmetric matrix
defining the cometric $g_x$ to emphasis that it is a tensor of
covariant type and operates with covectors. The matrix $g^{jk}_x$
is never invertible.

A differential manifold $M$ with a chosen subbundle $S$ of the tangent bundle and with a given nondegenerate ss-metric $Q$ on $S$ will be called the {\it sub-semi-Riemannian manifold} or shortly {\it ss-manifold}. If the index $\nu$ of $Q$ is $1$, then we call the triplet $(M,S,Q)$ a sub-Lorentzian manifold and in the case of $\nu=0$ we get the sub-Riemannian manifold widely studied in~\cite{Gromov,LiuSussmann,Montgomery, Strichartz} and numerous references therein.

We present a couple of examples of ss-manifolds.

$\mathbf{Example\;1}$.
\medskip

Let us consider the following example of sub-Lorentzian manifold that we call the Heisenberg group with sub-Lorentzian metric. This example was considered first in~\cite{Groch1, Groch2} and was also studied in~\cite{KM}. We remind
that the Heisenberg group $\mathbb H^1$ is the space $\mathbb R^3$
furnished with the non-commutative law of multiplication
$$(x,y,z)(x^{\prime},y^{\prime},z^{\prime})=\big(x+x^{\prime},y+y^{\prime},z+z^{\prime}+\frac{1}{2}(yx^{\prime}-xy^{\prime})\big).$$ This gives the $\mathbb R^3$ the structure of a non-abelian Lie group.
The two-dimensional horizontal bundle $S$ is given as
a span of left invariant vector fields
$$X=\frac{\partial}{\partial
x}+\frac{1}{2}y\frac{\partial}{\partial z},\qquad
Y=\frac{\partial}{\partial y}-\frac{1}{2}x\dfrac{\partial}{\partial
z},$$ that can be found as the left action of the Lie group. There is only one nonvanishing commutator $[X,Y]=Z=\frac{\partial}{\partial z}$. We suppose that the
Lorentzian metric $Q$ is defined on $S$ by setting
$$Q(X,X)=-1,\quad Q(Y,Y)=1,\quad Q(X,Y)=0.$$ Thus the triple
$(\mathbb R^3, S, Q)$ is called the Heisenberg group
with the Lorentzian metric, and to differ it from the classical case $\mathbb H^1$ we
use the notation $\mathbb H^1_L$. We say the classical case bearing in mind the manifold $(\mathbb R^3, S, d)$ with a positively definite metric $d$ on $S$.

The quadratic nondegenerate symmetric form $Q$ on $S$ is of the form
\begin{equation*}
    Q=\{Q_{ab}\}=\left(\begin{matrix}
           -1& 0\\
            0& 1
      \end{matrix}\right).
\end{equation*} Take the basis of Lie algebra associated with the Heisenberg group, considered as the Lie group,
$(X, Y, Z)\in T$. The dual basis of $T^*$ consists of the forms $dx$, $dy$, $\omega=dz-\frac{1}{2}(xdy-ydx)$. We wish to find the cometric $g=g^{jk}$. Let $g\,dx=a_1X+a_2Y$. Making use of~\eqref{eq:2.1} for $W$ replaced by $X$ and $Y$, we deduce that $g\,dx=-X$. In the same way we get $g\,dy=Y$. Thus, the equality $g(\zeta,\xi)=Q(g\zeta,g\xi)$ calculated for the basic forms implies the values $g=g^{jk}$ for $j,k=1,2$. The rest of the terms vanish because of $g(dx,\omega)=\langle g\,dx,\omega\rangle=0$, $g(dy,\omega)=\langle g\,dy,\omega\rangle=0$, and $g\omega=0$. Finally we get \begin{equation*}
   {g}^{jk}=\left(\begin{matrix}
           -1& 0 &0\\
            0& 1 &0\\
            0& 0 &0
      \end{matrix}\right).
\end{equation*}
\medskip

$\mathbf{Example\;2}$.
\medskip
Consider the example of ss-manifold related to the notion of Heisenberg-type groups based on quaternions~\cite{Markina,CDKR,Kap1}.
The manifold $M$ is $\mathbb R^7$. The vector fields
\begin{gather}
   X_1=\frac{\partial}{\partial x_1}+\frac{\displaystyle1}{\displaystyle2} \left(+x_2
   \frac{\partial}{\partial z_{1}}-x_4\frac{\partial}{\partial z_{2}}
   - x_3\frac{\partial}{\partial z_{3}}\right),
   \notag\\
   X_2=\frac{\partial}{\partial x_2}+\frac{\displaystyle1}{\displaystyle2} \left(-x_1
   \frac{\partial}{\partial z_{1}}-x_3\frac{\partial}{\partial z_{2}}
   + x_4\frac{\partial}{\partial z_{3}}\right),
   \notag\\
   X_3=\frac{\partial}{\partial x_3}+\frac{\displaystyle1}{\displaystyle2} \left(+x_4
   \frac{\partial}{\partial z_{1}}+x_2\frac{\partial}{\partial z_{2}}
   + x_1\frac{\partial}{\partial z_{3}}\right),
   \notag\\
   X_4=\frac{\partial}{\partial x_4}+\frac{\displaystyle1}{\displaystyle2} \left(-x_3
   \frac{\partial}{\partial z_{1}}+x_1\frac{\partial}{\partial z_{2}}
   - x_2\frac{\partial}{\partial z_{3}}\right),
   \notag
\end{gather} form the basis of four-dimensional horizontal distribution $S$. These vector fields come from the infinitesimal action of the noncommutative group law multiplication
$$
L_{(x,z)}(x',z')=(x,z)\circ (x',z')=\big(x+x',z+z'+\frac{1}{2}\im (\bar x*x')\big)
$$ for $(x,z)$ and $(x',z')$ from $\mathbb R^4\times\mathbb R^3$. Here $\im (\bar x*x')$ is the imaginary part of the product $\bar x*x'$ of the conjugate quaternion $\bar x$ to $x$ by another quaternion $x'$. See the details in Section~\ref{sec:6}. The distribution $S$ is bracket generating due to the commutation relations
\begin{gather}
    [X_1,X_2]=-Z_1,\quad [X_1,X_3]=Z_3,\quad [X_1,X_4]=Z_2,\notag\\
    [X_2,X_3]=Z_2,\quad [X_2,X_4]=-Z_3,\quad [X_3,X_4]=-Z_1,\notag
\end{gather} where $Z_{\beta}=\frac{\partial}{\partial z_{\beta}}$, $\beta=1,2,3$ form a basis of the complement to~$S$ in the tangent bundle.

We define the ss-metric $Q$ on $S$ by the matrix \begin{gather*}
    Q_{\alpha\beta}=\left(\begin{matrix}
           -1& 0 &0 &0\\
            0& -1 &0 &0\\
            0& 0 &1 &0\\
            0 &0 &0 &1\\
      \end{matrix}\right).
\end{gather*} The ss-metric $Q$ has index $2$. The corresponding cometric $g^{jk}$ is obtained like in the Example~1, has index $2$, and assumes the following form
\begin{gather*}
    g_{jk}=\left(\begin{matrix}
           -1& 0 &0 &0 &0 &0 &0\\
            0& -1 &0 &0 &0 &0 &0\\
            0& 0 &1 &0 &0 &0 &0\\
            0 &0 &0 &1 &0 &0 &0\\
            0 &0 &0 &0 &0 &0 &0\\
            0 &0 &0 &0 &0 &0 &0\\
            0 &0 &0 &0 &0 &0 &0
      \end{matrix}\right).
\end{gather*} More details about the manifold of the Example~2 the reader can find in Section~\ref{sec:6}.
\medskip

We refer to~\cite{ChangMarkinaVasiliev} for an example of sub-Lorentzian manifold based on the Lie group different from the nilpotent group.

\section{Christoffel symbols}\label{sec:3}

Recall that $S^{\bot}_x$
denotes the kernel of $g(x)$ and $S^{\bot}\subseteq T^*$, $S^{\bot}=\cup_{x\in M} S^{\bot}_x$. The space $S^{\bot}_x$ is the annihilator of $S_x$ in $T^*_x$.
From now on, we use the summation convention of the differential geometry.
\begin{lemma}\label{lem:2.11} Let $v$ be a covector from $S^{\bot}$. Then
$\langle v,Y\rangle=0$ for all $Y\in S$ if and only if $g^{jk}v_k=0$, $j=1,\ldots,n$.
\end{lemma}
\begin{proof} Let $\langle v,Y\rangle=0$ for any $Y\in S$. Pick up an arbitrary $Y\in S$, then
there exists a form $\omega$ such that $Y=g\omega$. Moreover, we can assume that $\omega\in T^*/S^{\bot}$. Indeed, if $\omega=\omega_1+v_1$, where $\omega_1\in T^*/S^{\bot}$ and $v_1\in S^{\bot}$, then for any $X\in S$ we get $$Q(X,g\omega)=Q(X,g\omega_1)+Q(X,gv_1)=Q(X,g\omega_1)+\langle v_1,X\rangle=Q(X,g\omega_1).$$  Thus
\begin{gather*}
0=\langle v,Y\rangle=\langle v,g\omega\rangle=Q(g\omega,gv)\quad \text{for all}\quad Y=g\omega\in S \quad  \Rightarrow\quad gv=0.
\end{gather*}
    Here we used the symmetry and the nondegeneracy of $Q$.

Conversely, having $gv=0$ we derive $0=\langle v,g\omega\rangle$ for any $\omega\in T^*/S^{\bot}$. Thus, $\langle v,Y\rangle=0$ for any $Y=g\omega$.
\end{proof}

\begin{lemma}\label{lemma:2.1}
   $(a)$ If $v$ is a section of the annihilator $S^{\bot}$, then
   $$g^{jk}\dfrac{\partial v_k}{\partial x^p}=-\dfrac{\partial g^{jk}}{\partial x^p}v_k.$$
$(b)$ If $x(t)$ is a curve in $M$ and $v(t)$ is such that $v(t)$ is a section of $S^{\bot}$ over $x(t)$, then
$$g^{jk}(x)\dot{v}_k=-\dfrac{\partial g^{jk}(x)}{\partial x^p}\dot{x}^pv_k$$ for all $t$ (here the dot denotes the $t$-derivative).\\
$(c)$ If $v$ and $w$ are sections of $S^{\bot}$, then $$\dfrac{\partial g^{jk}}{\partial x^p}v_kw_j=0.$$
\end{lemma}
\begin{proof}
 To prove $(a)$ one applies $\dfrac{\partial}{\partial x^p}$ to the identity $g^{jk}(x)v_k(x)=0$ which defines the null-bundle.

To prove $(b)$ we take the derivative $\dfrac{d}{dt}$ of the identity $g^{jk}(x(t))v_k(t)=0$.

Finally, to prove $(c)$ first we apply $(a)$ to obtain $$g^{jk}\dfrac{\partial v_k}{\partial x^p}=-\dfrac{\partial g^{jk}}{\partial x^p}v_k.$$ Then we multiply both sides by $w\in S^{\bot}$ and, making use of the symmetry of $g$, we get $$\frac{\partial g^{jk}}{\partial x^p}v_kw_j=
-g^{jk}\frac{\partial v_k}{\partial x^p}w_j=-g^{kj}w_j\frac{\partial v_k}{\partial x^p}=0$$ because of $g^{kj}w_j=0$.
\end{proof}

The following question can arise: how the information about the bracket generating properties reflects in the properties of $g$? If $X,\,Y\in S$, then $[X,Y]$ is an element of $T/S$. If $X\in S_x$ then there exists $\xi\in T^*_x$ with $X=g_x\xi$, and similarly $Y=g_x\eta$. The covectors $\xi$ and $\eta$ are not defined uniquely, as it was shown in the proof of Lemma~\ref{lem:2.11}. Thus they should be regarded as elements of $T^*_x/S^{\bot}_x$. The annihilator $S^{\bot}$ contains all the necessary information concerning the commutators through the pairing $\langle [X,Y], v\rangle$, when $v$ varies over $S^{\bot}$. Let us consider the trilinear form $\langle[g\xi,g\eta],v\rangle$ on $(T^*_x/S^{\bot}_x)\times(T^*_x/S^{\bot}_x)\times S^{\bot}_x$.
\begin{lemma}\label{lemma:2.2}
In local coordinates \begin{equation}\label{eq:2.21}\langle[g\xi,g\eta],v\rangle=\left(g^{jp}\dfrac{\partial g^{rq}}{\partial x^j}-g^{jq}\dfrac{\partial g^{rp}}{\partial x^j}\right)\xi_p\eta_qv_r\end{equation} for $v$ varying over $S^{\bot}_x$ for any $x$.
\end{lemma}
\begin{proof}
Let $\xi$ and $\eta$ denote any sections of $T^*$. Then $X^r=g^{rp}\xi_p$ and $Y^r=g^{rq}\eta_q$ are  sections of $S$, and
\begin{gather}
    [X,Y]^r=X^j\dfrac{\partial}{\partial x^j}Y^r-Y^j\dfrac{\partial}{\partial x^j}X^r\notag\\
    =g^{jp}\xi_p\dfrac{\partial g^{rq}}{\partial x^j}\eta_q-g^{jq}\eta_q\dfrac{\partial g^{rp}}{\partial x^j}\xi_p+g^{jp}\xi_pg^{rq}\dfrac{\partial \eta_q}{\partial x^j}-g^{jq}\eta_qg^{rp}\dfrac{\partial \xi_p}{\partial x^j}.\notag
\end{gather}
Taking the inner product with $v\in S^{\bot}$, we find that the last two terms are annihilated since $g^{rq}\dfrac{\partial \eta_q}{\partial x^j}v_r=g^{qr}v_r\dfrac{\partial \eta_q}{\partial x^j}=0$ and $g^{rp}\dfrac{\partial \xi_p}{\partial x^j}v_r=g^{pr}v_r\dfrac{\partial \xi_p}{\partial x^j}=0$. Thus we obtain~\eqref{eq:2.21}.
\end{proof}

We want do define the analogue of the Christoffel symbols but with the raised indexes and see the relation between them and the trilinear form defined in Lemma~\ref{lemma:2.2}. We write \begin{equation}\label{Gamma}\Gamma^{kpq}=\dfrac{1}{2}\left(g^{kj}\dfrac{\partial g^{pq}}{\partial x^j}-g^{pj}\dfrac{\partial g^{kq}}{\partial x^j}-g^{qj}\dfrac{\partial g^{kp}}{\partial x^j}\right).\end{equation} For sections $\xi\in T^*$ and $v\in S^{\bot}$ define $\Gamma(\xi,v)\in T$ by $\Gamma^k(\xi,v)=\Gamma^{kpq}\xi_pv_q$. In classical case of differential geometry the Christoffel symbols are used to express the covariant derivative in local coordinates. Unlike to the classical covariant derivative, which associates for two vector fields another vector field, the operator $\Gamma$, as we will see from the following lemma, associates a vector field for a pair of covector fields and, moreover, the resulting vector field is horizontal.

\begin{lemma}\label{lemma:2.3}
$\Gamma(\xi,v)$ is a well-defined vector field; that is it is independent of the choice of coordinates. Moreover, $\Gamma(\xi,v)$ is a horizontal vector field and $\Gamma(\xi+w,v)=\Gamma(\xi,v)$ for $w\in S^{\bot}$, so that $\Gamma\colon (T^*/S^{\bot})\times S^{\bot}\to S$.
\end{lemma}
\begin{proof}
Let us prove that $\Gamma^k(\xi,v)$ transforms as a tangent vector at each $x$: $\widetilde\Gamma^k(y)=\Gamma^d(x)\frac{\partial y^k}{\partial x^d}$, where $y=\psi(x)$ and $\psi$ is a local diffeomorphism determining a new coordinate system. By $\widetilde{g},\,\widetilde{\xi}$ and $\widetilde{v}$ denote the expressions for $g,\,\xi$ and $v$ in the new coordinates. We have
\begin{equation}\label{eq:zamena_perem}
\xi_k=\dfrac{\partial \psi^j(x)}{\partial x^k}\widetilde{\xi}_j,\quad v_k=\dfrac{\partial \psi^j(x)}{\partial x^k}\widetilde{v}_j,\quad
\widetilde{g}^{kj}(y)=g^{pq}(x)\dfrac{\partial \psi^k(x)}{\partial x^p}\dfrac{\partial \psi^j(x)}{\partial x^q}.
\end{equation}
In the new coordinates
\begin{gather}
\widetilde{g}^{jp}\dfrac{\partial \widetilde{g}^{kq}}{\partial y^j}=\left(g^{ab}\dfrac{\partial y^j}{\partial x^a}\dfrac{\partial y^p}{\partial x^b}\right)\dfrac{\partial }{\partial x^l}\left(g^{cd}\dfrac{\partial y^k}{\partial x^c}\dfrac{\partial y^q}{\partial x^d}\right)\dfrac{\partial x^l}{\partial y^j},\notag
\end{gather} hence
\begin{gather}
\widetilde{g}^{jp}\dfrac{\partial \widetilde{g}^{kq}}{\partial y^j}\widetilde{\xi}_p\widetilde{v}_q=\left[g^{ab}\dfrac{\partial y^j}{\partial x^a}\dfrac{\partial y^p}{\partial x^b}\dfrac{\partial g^{cd}}{\partial x^l}\dfrac{\partial y^k}{\partial x^c}\dfrac{\partial y^q}{\partial x^d}\dfrac{\partial x^l}{\partial y^j}\right.\notag
\end{gather}
\begin{equation}\label{eq:1}
\left.
+g^{ab}\dfrac{\partial y^j}{\partial x^a}\dfrac{\partial y^p}{\partial x^b}\dfrac{\partial x^l}{\partial y^j}g^{cd}\left(\dfrac{\partial^2 y^k}{\partial x^l\partial x^c}\dfrac{\partial y^q}{\partial x^d}+\dfrac{\partial y^k}{\partial x^c}\dfrac{\partial^2 y^q}{\partial x^l\partial x^d}\right)\right]\widetilde{\xi}_p\widetilde{v}_q.
\end{equation}
The first term of~\eqref{eq:1} equals to $$\dfrac{\partial
y^k}{\partial x^c}\left(g^{ab}\dfrac{\partial y^j}{\partial
x^a}\dfrac{\partial g^{cd}}{\partial x^l}\dfrac{\partial
x^l}{\partial y^j}\right)\xi_bv_d= \left(\dfrac{\partial
y^k}{\partial x^c}\right)\left(g^{ab}\dfrac{\partial
g^{cd}}{\partial x^a}\right)\xi_bv_d.$$ Changing indexes $b$ to
$p$, $d$ to $q$, and $a$ to $j$ we recognize the tangent bundle
transformation of $g^{jp}\dfrac{\partial g^{kq}}{\partial
x^j}\xi_pv_q$. The middle term vanishes since
\begin{equation}\label{eq:vanish}g^{cd}\dfrac{\partial
y^q}{\partial x^d}\widetilde{v}_q=g^{cd}v_d=0\end{equation} and
the last term gives
\begin{equation}\label{eq:gamma1}g^{lb}\xi_bg^{cd}\dfrac{\partial
y^k}{\partial x^c}\dfrac{\partial^2 y^q}{\partial x^d\partial
x^l}\widetilde v_q.\end{equation} The middle term in
$\Gamma^{kpq}\xi_pv_q$ transforms as follows
$$\widetilde{g}^{jq}\dfrac{\partial \widetilde{g}^{kp}}{\partial
y^j}\widetilde{\xi}_p\widetilde{v}_q=\left(\dfrac{\partial
y^k}{\partial x^c}\right)\left(g^{ab}\dfrac{\partial
g^{cd}}{\partial x^a}\right)\xi_dv_b=\left(\dfrac{\partial
y^k}{\partial x^c}\right)\left(g^{jq}\dfrac{\partial
g^{cp}}{\partial x^j}\right)\xi_pv_q.$$ The other terms vanish by
the same reason as in~\eqref{eq:vanish}. The third term in
$\Gamma^{kpq}\xi_pv_q$ in the new coordinates takes the form
\begin{equation}\label{eq:366}
-\widetilde{g}^{jk}\dfrac{\partial \widetilde{g}^{pq}}{\partial y^j}\widetilde{\xi}_p\widetilde{v}_q=-\dfrac{\partial y^k}{\partial x^c}\left(
g^{jc}\dfrac{\partial g^{pq}}{\partial x^j}\right)\xi_pv_q-g^{lb}\dfrac{\partial y^k}{\partial x^b}g^{cd}\xi_c\dfrac{\partial^2 y^q}{\partial x^d\partial x^l}\widetilde{v}_q.
\end{equation} We see that the last term from~\eqref{eq:366} is canceled with~\eqref{eq:gamma1} (after the change of indexes). Taking together the rest of terms, we get the desired transformation law
$$\widetilde{\Gamma}^{kpq}\widetilde{\xi}_p\widetilde{v}_q=\dfrac{\partial y^k}{\partial x^c}\Gamma^{cpq}\xi_pv_q.$$

To show that $\Gamma(\xi,v)\in S$ we take a covector $\omega\in S^{\bot}_x$ and calculate $$\langle\Gamma(\xi,v),\omega\rangle=\Gamma^{kpq}\xi_pv_q\omega_k.$$ Using~\eqref{Gamma} and Lemma~\ref{lemma:2.1}, we argue for each term of $\Gamma^{k}(\xi_pv_q)$ as it follows $${g}^{jp}\dfrac{\partial {g}^{kq}}{\partial x^j}\xi_pv_q\omega_k=-g^{jp}\xi_pg^{kq}\frac{\partial v_q}{\partial x^j}\omega_k=-g^{jp}\xi_pg^{qk}\omega_k\frac{\partial v_q}{\partial x^j}=0$$ and get $\langle\Gamma(\xi,v),\omega\rangle=0$, that implies $\Gamma(\xi,v)\in S$.

The property $\Gamma(\xi+\omega,v)=\Gamma(\xi,v)$ for $\omega\in S^{\bot}_x$ follows from
\begin{equation}\label{eq:2.31}
\Gamma^k(\xi,v)=\frac{1}{2}\Big(g^{jp}\xi_pg^{kq}\frac{\partial v_q}{\partial x_j}+g^{jk}\frac{\partial g^{pq}}{\partial x^j}\xi_pv_q\Big)
\end{equation}

and Lemma~\ref{lemma:2.1}.
\end{proof}

Analogously to sub-Riemannian situation~\cite{Strichartz} we have
\begin{theorem}\label{th:2.4}
A vector field $X\in S$ is a $2$-step bracket generator if and only if $\Gamma(\xi,\cdot)\colon S^{\bot}\to S$ is injective, where $X=g\xi$. In particular, $S$ satisfies the $2$-step bracket generating hypothesis if and only if $\,\Gamma(\xi,\cdot )\colon S^{\bot}\to S$ is injective for every nonzero form $\xi\in T^*/ S^{\bot}$.
\end{theorem}
\begin{proof} In the proof we exploit the properties of different linear mappings which we defined up to now. We have \begin{equation}\label{eq:gamma}
   \langle [g\xi,g\eta],v\rangle=2\Gamma(\xi,v)\eta
\end{equation}  by~\eqref{eq:2.21} and~\eqref{eq:2.31}.
In order to show that the vector field $X$ is a $2$ step bracket generator we must show that the vector fields $[X,Y]\mod S$ fill out all $T/ S$ (at each $x$) as $Y$ varies over $S$. In other words, the mapping
\begin{equation}\label{eq:c1}
[X,\cdot]\mod S\ \colon S\to T/ S\quad\text{is surjective at each}\ \ x.
\end{equation} Since at any $x$ the space $T^*/ S^{\bot}$ is canonically isomorphic to the dual to $S$, statement~\eqref{eq:c1} is equivalent to
\begin{equation}\label{eq:c2}
[g\xi,g(\cdot)]\colon T^*/ S^{\bot}\to T/ S\quad\text{is surjective at each}\ \ x, \quad\text{where}\ \ X=g\xi.
\end{equation} We notice that at any $x$ the space $S^{\bot}$ is canonically isomorphic to the dual to $T/ S$. Thus~\eqref{eq:c2} is equivalent to
\begin{equation}\label{eq:c3}
\Gamma(\xi,\cdot)\colon S^{\bot}\to S\quad\text{is injective at each}\ \ x.
\end{equation}
\end{proof}

We discussed earlier the relation between classical notion of covariant derivative and the Christoffel symbols. The closest notion to the notion of covariant derivative is symmetrized covariant derivative that was defined in~\cite{Strichartz}. It is natural to define the same concept on ss-manifolds.
\begin{definition}
The symmetrized covariant derivative $\bigtriangledown_{\sym}$ of a vector field $Y$ is defined by
\begin{equation}\label{eq:cov_der}
 (\bigtriangledown_{\sym} Y)^{kq}=g^{kj}\dfrac{\partial Y^q}{\partial x^j}+g^{qj}\dfrac{\partial Y^k}{\partial x^j}-Y^j\dfrac{\partial g^{kq}}{\partial x^j}.
\end{equation}
\end{definition} Thus $(\bigtriangledown_{\sym})_x\colon T_x\to T_x\times T_x$.

\begin{lemma}
$\bigtriangledown_{\sym}$ is a well-defined differential operator from tensors of type $(1,0)$ to symmetric tensors of type $(2,0)$. Furthermore, if $Y$ is a vector field from $S$, that is $Y=g\xi$, then $(\bigtriangledown_{sym}Y)^{kq}v_q=2\Gamma^k(\xi,v)$ for any $v\in S^{\bot}$.
\end{lemma}
\begin{proof}
The symmetry follows from the symmetry of the cometric $g$. Let us show that $\bigtriangledown_{\sym}Y$, $Y\in T$, transforms as a tensor field of rank $(2,0)$.
We check how the first term of \eqref{eq:cov_der}
transforms with the following change of coordinates:
$$\widetilde{Y}^{k}=\frac{\partial y^k}{\partial x^j}Y^j,\quad\widetilde{g}^{kj}=g^{pq}\frac{\partial y^k}{\partial x^p}\frac{\partial y^j}{\partial x^q}.$$
\begin{eqnarray*}
\widetilde{g}^{kj}\dfrac{\partial \widetilde{Y}^q}{\partial y^j} & = &
g^{\alpha\beta}\frac{\partial y^k}{\partial x^{\alpha}}\frac{\partial y^j}{\partial x^{\beta}}
   \frac{\partial}{\partial x^l}\left(\frac{\partial y^q}{\partial x^a}Y^a\right)\frac{\partial x^l}{\partial y^j}
\\
& = & g^{\alpha\beta}\frac{\partial y^k}{\partial x^{\alpha}}\frac{\partial y^j}{\partial x^{\beta}}\frac{\partial x^l}{\partial y^j}
   \frac{\partial^2 y^q}{\partial x^l \partial x^a}Y^a
+
g^{\alpha\beta}\frac{\partial y^k}{\partial x^{\alpha}}\frac{\partial y^j}{\partial x^{\beta}}\frac{\partial x^l}{\partial y^j} \frac{\partial Y^a}{\partial x^l}\frac{\partial y^q}{\partial x^a}
\\
& = & g^{\alpha l}\frac{\partial y^k}{\partial x^{\alpha}}\frac{\partial^2 y^q}{\partial x^l \partial x^a}Y^a
+
g^{\alpha l}\frac{\partial y^k}{\partial x^{\alpha}}\frac{\partial Y^a}{\partial x^l}\frac{\partial y^q}{\partial x^a}.
\end{eqnarray*}
Analogously, the second term
\begin{gather*}
   \widetilde{g}^{qj}\dfrac{\partial \widetilde{Y}^k}{\partial y^j}
   =g^{\alpha l}\frac{\partial y^q}{\partial x^{\alpha}}\frac{\partial^2 y^k}{\partial x^l \partial x^a}Y^a+g^{\alpha l}\frac{\partial y^q}{\partial x^{\alpha}}\frac{\partial Y^a}{\partial x^l}\frac{\partial y^k}{\partial x^a}.
\end{gather*}
And the third term
\begin{eqnarray*}
\widetilde{Y}^j\dfrac{\partial \widetilde{g}^{kq}}{\partial y^j}
&  = & \frac{\partial y^j}{\partial x^{\alpha}}Y^{\alpha}\frac{\partial }{\partial x^{l}}\left(g^{ab}\frac{\partial y^k}{\partial x^a}
   \frac{\partial y^q}{\partial x^b}\right)\frac{\partial x^l}{\partial y^j}\\
& = & \frac{\partial y^j}{\partial x^{\alpha}}Y^{\alpha}\frac{\partial g^{ab}}{\partial x^l}\frac{\partial y^k}{\partial x^a}
   \frac{\partial y^q}{\partial x^b}\frac{\partial x^l}{\partial y^j}\\
& + &
\frac{\partial y^j}{\partial x^{\alpha}}Y^{\alpha}g^{ab}\frac{\partial^2 y^k}{\partial x^l \partial x^a} \frac{\partial y^q}{\partial x^b}
\frac{\partial x^l}{\partial y^j}+\frac{\partial y^j}{\partial x^{\alpha}}Y^{\alpha}g^{ab}\frac{\partial y^k}{\partial x^a}
   \frac{\partial^2 y^q}{\partial x^l\partial x^b}\frac{\partial x^l}{\partial y^j}.
\end{eqnarray*}
After summation and necessary renaming of indexes we have
$$(\bigtriangledown_{\sym} \widetilde{Y})^{kq}= \dfrac{\partial y^k}{\partial x^i}\dfrac{\partial y^q}{\partial x^j}(\bigtriangledown_{\sym} Y)^{ji}.$$

Let us show the second statement of the theorem. We assume $Y=g\xi\in S$ and $v\in S^{\bot}$. Then from~\eqref{eq:gamma} follows that
\begin{eqnarray*}
    (\bigtriangledown_{\sym} g\xi)^{kq}v_q & = & \Big(g^{kj}\frac{\partial g^{pq}}{\partial x^j}\xi_p+g^{kj}g^{pq}\frac{\partial \xi_p}{\partial x^j}\Big)v_q
\\
& + & \Big(g^{qj}\frac{\partial g^{pk}}{\partial x^j}\xi_p+g^{qj}g^{pk}\frac{\partial \xi_p}{\partial x^j}\Big)v_q
\\
& - & g^{jp}\xi_p
\frac{\partial g^{kq}}{\partial x^j}v_q
\\
& = & g^{jp}\xi_pg^{kq}\dfrac{\partial v_q}{\partial x^j}+g^{kj}\dfrac{\partial g^{pq}}{\partial x^j}\xi_pv_q
=2\Gamma^k(\xi,v)
\end{eqnarray*}
by~\eqref{eq:2.31} and $gv=0$.
\end{proof}

\section{Hamiltonian system, exponential map and lengths of curves}\label{sec:4}

The distribution $S_x$ at each point $x$ of ss-manifold $M$ has the structure of $\mathbb R^m$ equipped with a nondegenerate metric $Q_x$ of index $\nu$. The presence of the nondegenerate metric yields the following trichotomy.
\begin{definition}
A horizontal tangent vector $w\in S_x$ is
\begin{itemize}
\item []{{\it spacelike} \qquad\quad if \quad $Q_x(w,w)>0$\quad or \quad $w=0$,}
\item []{{\it null}  \qquad\qquad\quad if \quad $Q_x(w,w)=0$\quad and \quad $w\neq 0$,}
\item []{{\it timelike}  \ \qquad\quad if \quad $Q_x(w,w)<0$,}
\item []{{\it nonspacelike} \ \quad if \ it is either timelike or null.}
\end{itemize}
\end{definition}
The set of all null vectors in $S_x$ is called {\it null-cone} at $x\in M$. The category into which a given tangent vector falls is called its {\it causal character}. The terminology is adapted from the relativity theory, and particularly in the Lorentz case, null-vectors are called {\it lightlike}. For the nice and complete presentation of the semi-Riemannian geometry see~\cite{Oneill}.

The covectors $\xi(x)\in T^*_x/ S^{\bot}_x$ receive the same causal structure according to the values of $\langle g_x\xi(x),\xi(x)\rangle$. The covectors $v\in S^{\bot}(x)$ we shall call annihilators to distinguish them from the null-covectors.

\begin{definition}
A horizontal tangent vector field $X\in S$ is spacelike, null or timelike if at each point $x\in M$ the vector $X(x)$ is spacelike, null or timelike respectively.
\end{definition}

\begin{definition}
A section $\xi\in T^*/ S^{\bot}$ is spacelike, null or timelike if at each point $x\in M$ the covector $\xi(x)$ is spacelike, null or timelike respectively.
\end{definition}

As we mentioned from the beginning, we work with the special class
of admissible curves that tangent to the distribution $S$ and that
we called horizontal curves. We borrow this name from the
sub-Riemannian geometry. We say that a horizontal curve $c(s)$ is
spacelike, null or timelike if the tangent vector $\dot c(s)$ is
spacelike, null or timelike respectively at each point of $c(s)$
where it exists. We can give the definition of the spacelike, null
or timelike curve using the causal structure of the cotangent
space $T^*$ according to the sign of $\langle
g_{c(s)}\xi(s),\xi(s)\rangle=Q_{c(s)}(\dot c(s),\dot c(s))$, where
$\dot c(s)=g_{c(s)}\xi(s)$. We call a horizontal curve the {\it
causal} if the tangent vector $\dot c(s)$ (the covector $\xi(s)$)
is nonspacelike

In the sub-Lorentzian case we also introduce (as in the classical Lorentz manifolds) the time orientation.
\begin{definition}
A time orientation on $(M,S,Q)$ is a continuous horizontal timelike section $\mathcal T$ os $S$.
\end{definition} If $M$ admits a time orientation $\mathcal T$, then $\mathcal T$ divides all nonspacelike horizontal vectors into two disjoint classes, called future directed and past directed. Namely, nonspacelike $w\in S_x$ is said to be {\it future} (respectively {\it past}) {\it directed} if $Q_x(\mathcal T(x),w)<0$ (respectively $Q_x(\mathcal T(x),w)>0$). We assume that any considered in the article sub-Lorentzian manifold $(M,S,Q)$ will be time oriented.

Since $g\colon T^*/S^{\bot}\to S$ is injective the time orientation can be brought to $T^*/S^{\bot}$.

\begin{definition} The globally defined section $\tau\in T^*$ such that $\mathcal T=g\tau$ is time orientation on $T^*/S^{\bot}$.
\end{definition} The covectors from $S^{\bot}$ we can consider as null-covectors.

The notion of arc length of a curve segment in Euclidean space
generalizes in a natural way to ss-manifolds. Since the term ``arc
length'' can be misleading since, for example, a null-curve has
length zero. Therefore, we use the name ``natural parameter'' in
stead of ``arc length''.
\begin{definition}\label{natpar}
Let $c:[a,b]\to M$ be a piecewise smooth curve segment in a ss-manifold $(M,S,Q)$. The natural parameter of $c(s)$ is $$L(c)=\int_{a}^{b}|Q(\dot c(s),\dot c(s))|^{1/2}\,ds.$$
\end{definition} As in the classical case it can be shown that
\begin{itemize}
\item [$\bullet$]{the natural parameter is not changing under the monotone reparameterization and}
\item [$\bullet$]{if $c(s)$ is a curve segment with $|\dot c(s)|=|Q(\dot c(s),\dot c(s))|^{1/2}>0$, there is a strictly increasing reparameterization function $h$ such that $\gamma=c(h)$ has $|\dot\gamma|=1$.}
\end{itemize} In the latter cases $\gamma$ is said to have {\it unit speed} or {\it natural reparameterization}.

Now we define the extremal using the Hamilton function. Given the cometric $g_x\colon T^*_x\to S_x$
we form the Hamiltonian function
\begin{equation}\label{hamfun1}
H(x,\xi)=\frac{1}{2}\langle g_x(\xi),\xi\rangle
\end{equation} on $T^*_x$. To emphasize the dependence of the cometric on $x$ we write $g(x)$ instead of $g_x$ when it is necessary. If we have the orthonormal basic $X_1,\ldots,X_{\nu},\ldots,X_m$ of $S$ we can write the Hamiltonian function in the form
\begin{equation}\label{hamfun2}
H(x,\xi)=-\frac{1}{2}\sum_{j=1}^{\nu}\langle X_j(x),\xi\rangle^2+\frac{1}{2}\sum_{j=\nu+1}^{m}\langle X_j(x),\xi\rangle^2,
\end{equation}
where $\nu$ is the index of $g_x$. Consider the Hamiltonian equations
\begin{equation*}
\dot x(s)=\nabla_{\xi}H(x,\xi),\qquad \dot\xi(s)=-\nabla_xH(x,\xi)
\end{equation*} that explicitly can be expressed as
\begin{eqnarray}\label{eq:ham}
\dot x^k(s) & = & g^{kj}(x(s))\xi_j(s),\qquad k=1,\ldots,n,\nonumber  \\
\dot\xi_k(s) & = & -\frac{1}{2}\frac{\partial g^{pq}(x(t))}{\partial x^k}\xi_p(s)\xi_q(s).
\end{eqnarray} An absolutely continuous curve $\Gamma(s)$ on $M$ satisfying~\eqref{eq:ham} is called a {\it characteristic} of $H$. In this paper we will consider only the bicharacteristics $\Gamma(s)$ such that $H(\Gamma(s))=H(x(s),\xi(s))\neq 0$ that are called in literature the {\it normal biextremals}. The detailed discussion of the structures of normal and abnormal geodesics see, for instance,~\cite{LiuSussmann, Montgomery2, Agrachev}. Since we work only with normal biextremals we will drop the word ``normal'' for shortness. If $H\in C^1(T^*)$ then an extremal, of $H$ is a curve $x(s)$ which is a projection on manifold of some biextremal $\Gamma(s)$ of~$H$. The bicharacteristics of a Hamiltonian $H\in C^k(T^*)$ are curves of class $C^k$ along which $H$ is constant. In this case it means that an extremal has a parametrization by the natural parameter. The next result is the consequence of this.

\begin{proposition}\label{prop:3.1}
If $\gamma\colon[a,b]\to M$ is a normal extremal, then either $Q_{\gamma(s)}(\dot\gamma(s),\dot\gamma(s))<0$ or $Q_{\gamma(s)}(\dot\gamma(s),\dot\gamma(s))=0$ or $Q_{\gamma(s)}(\dot\gamma(s),\dot\gamma(s))>0$ for all $s\in[a,b]$. Moreover, if $\gamma$ is nonspacelike in the sub-Lorentzian manifold, then it does not change its orientation.
\end{proposition}

\begin{proof}
We have $$\frac{1}{2}Q_{\gamma(s)}(\dot\gamma(s),\dot\gamma(s))=\frac{1}{2}\langle \xi(s),g_{\gamma(s)}\xi(s)\rangle=H(\gamma(s),\xi(s))$$ which is constant along $\gamma$. The orientation preserving property of a smooth curve is obvious.
\end{proof} It is possible to reformulate Proposition~\ref{prop:3.1} in terms of cometric $g$.
\begin{proposition}\label{prop:3.2}
If $\gamma\colon[a,b]\to M$ is a normal extremal, then either $\langle \xi(s),g_{\gamma(s)}\xi(s)\rangle<0$ or $\langle \xi(s),g_{\gamma(s)}\xi(s)\rangle=0$ or $\langle \xi(s),g_{\gamma(s)}\xi(s)\rangle>0$ for all $s\in[a,b]$.
\end{proposition}

Let us define the energy for the curve $c:[a,b]\to M$ by
\begin{equation}\label{eq:ener}
E(c)=\int_{a}^{b}|Q(\dot c(s),\dot c(s))|\,ds.
\end{equation} In semi-Riemannian geometry extremals $\gamma(s)$ are defined as curves which have parallel tangent vector field $(\dot\gamma)$ or, equivalently, which have the acceleration zero: $\ddot{\gamma}(s)=0$. It is true
that semi-Riemannian extremals lift to solutions of~\eqref{eq:ham} on the cotangent bundle. Thus the definition of extremals like Hamilton extremals is correct generalization. For the sub-Riemannian and sub-Lorentzian cases see~\cite{Groch3, Montgomery, Strichartz}. Also, if we formulate the variational problem of minimizing energy $E(c)$ over all smooth horizontal curves joining points $p$ and $q$ in $M$
then the associated Euler equation is~\eqref{eq:ham}. Notice also that if we differentiate the first equation and substitute the second we obtain \begin{equation}\label{eq:geod}\ddot x^k(s)+\Gamma^k(\xi,\xi)=0\end{equation} which is the analogue of the equation of the extremals in semi-Riemannian geometry. Notice, that we can not solve~\eqref{eq:geod} for $\xi$ in terms of $x$ in any way. Thus~\eqref{eq:geod} does not reduce to the equation in $x$ alone. Neither~\eqref{eq:geod} together with $\dot x=g\xi$ is equivalent to~\eqref{eq:ham}.

Given $p\in M$, $u\in T^*_p$, and the coordinate system with the origin at $p$, the existence and uniqueness theorem for ordinary differential equations guarantees that the solution exists and is unique on an interval around zero provided the initial conditions $x(0)=p$, $\xi(0)=u$. As on sub-Riemannian manifolds the solution to~\eqref{eq:ham} can be continued as long as $x(t)$ remains in~$M$.

\begin{lemma}
Let $x(s)$ be a normal extremal for $0\leq s<a$ and suppose $x(s)$ remains inside a compact subset of $M$. Then $x(s)$ can be extended beyond $s=a$.
\end{lemma}

\begin{proof}
Over the compact set $K\subset M$ choose an orthonormal basis $v^{(1)}(x),\ldots,v^{(n-k)}(x)$ of $S^{\bot}_x$ and complete to an orthonormal basis of $T^*_x$ by adding $u^{(1)}(x),\ldots,u^{(k)}(x)$. By definition all sections are smoothly varying on the compact set and hence bounded. Then the section $\xi(s)$ along the extremal $x(s)$ can be written as
\begin{equation}\label{eq:xi}
 \xi(s)=\sum_{j=1}^{m}a_j(s)u^{(j)}(x(s))+\sum_{l=1}^{n-m}b_l(s)v^{(l)}(x(s)),
\end{equation} where $m$ is the rank of $S$. Consider~\eqref{eq:ham} as a system of equations for $x(s)$, $a_j(s)$, and $b_l(s)$. The functions $x^k(s)$ are uniformly bounded on $K$. Let us show that the functions $a_j(s)$ and $b_l(s)$ are also bounded. We have $$H(x(s),\xi(s))=\frac{1}{2}\sum_{j=1}^{m}\langle gu^{(j)},u^{(k)}\rangle a_j(s)a_k(s)$$ by~\eqref{eq:xi}. Since extremals do not change the causal character and the Hamilton is constant along them, the value of the matrix $g\big(u^{(j)}(x(s)),u^{(k)}(x(s))\big)$ is bounded from zero on $K$. It follows that $a_j(s)$ are uniformly bounded along extremals. Let us show that $b_l(s)$ are bounded. We write $\xi_k(s)$ as
\begin{equation}\label{eq:xik}
 \xi_k(s)=\sum_{j=1}^{m}a_j(s)u^{(j)}_k(x(s))+\sum_{l=1}^{n-m}b_l(s)v^{(l)}_k(x(s)),
\end{equation} where $u^{(j)}_k$ and $v^{(l)}_k$ are coordinates of $u^{(j)}$ and $v^{(l)}$ in the local chart coordinates. We substitute~\eqref{eq:xik} in the second equation of~\eqref{eq:ham} and take into account the first one also. Notice, that the terms involving $gv$ and $\dfrac{\partial g^{pq}}{\partial x^k}v^{(l)}_pv^{(m)}_q$ vanish by Lemma~\ref{lemma:2.1} $c)$ since $v\in S^{\bot}$. Finally, we get
\begin{eqnarray*}
\sum_{j=1}^{m}\dot a_ju^{(j)}_k &  + & \sum_{j=1}^{n-m}\dot b_jv^{(j)}_k+\sum_{j=1}^{m}\sum_{l=1}^{}a_j\frac{\partial u^{(j)}_{k}}{\partial x^r}g^{rp}a_lu^{(l)}_p+
\sum_{j=1}^{n-m}\sum_{l=1}^{m}b_j\frac{\partial v^{(j)}_{k}}{\partial x^r}g^{rp}a_lu^{(l)}_p\\
& = & -\frac{1}{2} \frac{\partial g^{pq}}{\partial x^k}\Big(\sum_{j=1}^{m}a_ju^{(j)}_p\sum_{l=1}^{m}a_lu^{(l)}_q+\sum_{j=1}^{m}a_ju^{(j)}_p\sum_{l=1}^{n-m}b_lv^{(l)}_q+\sum_{j=1}^{n-m}b_jv^{(j)}_p\sum_{l=1}^{m}a_lu^{(l)}_q\Big).
\end{eqnarray*} If we dot both sides of equation with $v^{(1)}_k,v^{(2)}_k,\ldots,v^{(n-m)}_k$, then we obtain the linear system $$\dot b=Ab+C,\quad b=(b_1,\ldots,b_{n-m}),$$ where the matrix $A$ and the vector function $C$ linearly depend on bounded functions $a_j$, $u^{(j)}$, $v^{(j)}$ and hence $A$ and $C$ are bounded. The linear system of the first order differential equations with bounded coefficients has bounded solution. We conclude that $b_j$ are bounded for $j=1,\ldots,n-m$. Thus all the functions $x^k(t)$ and $\xi_k(t)$ are uniformly bounded, and the local existence theorem implies the solution of~\eqref{eq:ham} extends.
\end{proof}

Now we can define the exponential map.
\begin{definition}
If $p\in M$, let $D_p$ be the set of covectors $w$ in $T^*_p$ such that the extremal $x_{w}(s)$ is defined at least on $[0,1]$ and $x(0)=p$, $\xi(0)=w$. The exponential map of $M$ at $p$ is the function $$\exp_p\colon D_p\to M,\quad\text{such that}\quad \exp_p(w)=x_w(1).$$
\end{definition} The set $D_p$ is the largest subset of $T^*_p$ on which $\exp_p$ can be defined. Fix $w\in T^*_p$ and $\tau\in \mathbb R$. Then the extremal
 $s\to x_w(s\tau)$ is such that $\tau\xi(0)=\tau w$. Hence $x_{\tau w}(\tau)=x_{w}(s\tau)$ for all $\tau$ and $s$ where the both sides are well defined. Particularly $$\exp_p(sw)=x_{sw}(1)=x_{w}(s).$$

As in the sub-Riemannian geometry the exponential map is always
differentiable, since the solution of the Hamiltonian system
depends smoothly on the initial data. But the exponential map is
not a diffeomorphism at the origin. The reason is that all the
extremals emanating from $p$ must have tangent vectors in $S_p$,
but for any annihilator $v\in S^{\bot}$ we have
$$\exp_p(v)=x_v(1)=p,\quad\text{since}\quad\dot x^k=0\ \
\text{by}~\eqref{eq:ham}.$$ We prove the following analogue of
Gauss lemma. In lemma we use the identification of a cotangent
space $T^*_p$ at $p$ with the tangent to $T^*_p$ space
$T_u(T^*_p)$ at point $u\in T^*_p$. The covector $w\in T^*_p$ at
point $u\in T^*_p$ is identified with the vector $w\in
T_u(T^*_p)$. The radial vector $r\in T_u(T^*_p)$ means that it is
a scalar multiple of a covector $u\in T^*_p$.
\begin{lemma}\label{gauss}
Let $u$ be a cotangent
vector in $T^*_p$ such that $u\neq 0$ and lies inside $D_p$. Let $r$ be a radial vector and $w$ be any other covector at point $u\in T^*_p$. Then
\begin{itemize}
\item [(i)]{$$\langle g_pr,w\rangle=Q_{\exp_pu}\big(d(\exp_p)_uw,d(\exp_p)_ur\big)$$ provided $d(\exp_p)_uw\in S_{\exp_pu}$}
\item [(ii)]{$$\langle g_pr,w\rangle=\langle d(\exp_p)_uw,\xi \rangle$$ where $\xi$ is a cotangent lift of the extremal $t\mapsto \exp_p(tu)$ at $t=1$.}
\end{itemize}
\end{lemma}

\begin{proof} Let us prove (i). Since $r$ is radial, we can assume $r=u$.
Take the curve $v(s)=u+sw$ in~$T^*_p$. Let us suppose that the
exponential mapping is defined in the cylindrical neighborhood
$D_p\times [0,1]$. Consider the parameterized surface $x: A\to M$,
$A=\{(t,s): 0\leqslant t\leqslant1,-\varepsilon<s<\varepsilon\}$
given by $x(t,s):=\exp_p(t(u+sw))$. Note that $$\dfrac{\partial
\big(t(u+sw)\big)}{\partial s}(1,0)=w,\quad \dfrac{\partial
\big(t(u+sw)\big)}{\partial t}(1,0)=u,$$ and the curves $t\mapsto
x(t,s)$ are extremals for any fixed $s$ starting from the point
$x(0,s)=p$ with the initial covectors $u+sw$. Then
$$\dfrac{\partial x}{\partial s}(1,0)=d(\exp_p)_u\dfrac{\partial t(u+sw)}{\partial s}(1,0)=d(\exp_p)_uw,$$
$$\dfrac{\partial x}{\partial t}(1,0)=d(\exp_p)_u\dfrac{\partial t(u+sw)}{\partial t}(1,0)=d(\exp_p)_uu.$$
Thus, we need to show $\langle g_pu,w\rangle=Q_{\exp_pu}\big(\dfrac{\partial x}{\partial s},\dfrac{\partial x}{\partial t}\big)(1,0)$. Let $\xi(t,s)$ be a cotangent lift of the extremal $t\mapsto \exp_p(tv(s))$, particularly $\xi(1,s)=u+sw$. Then $Q_{\exp_pu}\big(\dfrac{\partial x}{\partial s},\dfrac{\partial x}{\partial t}\big)(1,0)=\langle \dfrac{\partial x}{\partial s},\xi\rangle(1,0)$ and our main aim becomes to show
\begin{equation}\label{eq:main}
\langle g_pu,w\rangle=\langle \dfrac{\partial x}{\partial s},\xi\rangle(1,0).
\end{equation}
We denote $f(t,s)=\langle \dfrac{\partial x}{\partial s},\xi\rangle(t,s)$ and calculate the derivative $\dfrac{\partial f}{\partial t}(t,0)$.
We have
\begin{gather*}
\frac{\partial f}{\partial t}(t,s)=\frac{\partial\xi_k}{\partial t}(t,s)\dfrac{\partial x^k}{\partial s}(t,s)+\xi_p(t,s)\dfrac{\partial^2x^p}{\partial s\partial t}(t,s).
\end{gather*}
Replacing $\dfrac{\partial\xi_k}{\partial t}(t,s)$ and $\dfrac{\partial x^p}{\partial t}(t,s)$ from the Hamilton-Jacobi equations \eqref{eq:ham} we obtain
\begin{eqnarray}\label{eq:333}
\frac{\partial f}{\partial t}(t,s) & = & -\dfrac{1}{2}\dfrac{\partial g^{pq}(x)(t,s)}{\partial x^k}\xi_p(t,s)\xi_q(t,s)\dfrac{\partial x^k}{\partial s}(t,s)+\xi_p(t,s)\dfrac{\partial}{\partial s}(g^{pq}(x(t,s))\xi_q(t,s))\nonumber \\
& = & \frac{\partial}{\partial s}\Big(\frac{1}{2}\langle g(x)\xi(t,s),\xi(t,s)\rangle\Big)\qquad\text{for any}\ \ t\ \ \text{and}\ \ s.
\end{eqnarray} Since the Hamilton $\frac{1}{2}\langle g(x)\xi(t,s),\xi(t,s)\rangle$ is constant along the extremal, then~\eqref{eq:333} can be written as $$\frac{\partial f}{\partial t}(t,s)=\frac{\partial}{\partial s}\Big(\frac{1}{2}\langle g(x)\xi(t,s),\xi(t,s)\rangle\Big)=\frac{\partial}{\partial s}\Big(\frac{1}{2}\langle g_p(u+sw),(u+sw)\rangle\Big)\qquad\text{for any}\ \ t.$$ Then $$\frac{\partial f}{\partial t}(t,0)=\frac{\partial}{\partial s}\Big(\frac{1}{2}\langle g_p(u+sw),(u+sw)\rangle\Big)(t,0)=\langle g_pu,w\rangle\qquad\text{for any}\ \ t.$$
We have $$f(0,0)=\langle \dfrac{\partial x}{\partial s},\xi\rangle(0,0)=\langle d(\exp_p)_u tw,u\rangle(0,0)=0$$ that implies $$f(t,0)=t\langle g_pu,w\rangle\quad\Longrightarrow\quad f(1,0)=\langle \dfrac{\partial x}{\partial s},\xi\rangle(1,0)=\langle g_pu,w\rangle.$$ We proved~\eqref{eq:main} and hence (i).

To prove (ii) we argue in a similar way. Take the curve
$v(s)=u+sw$ in $T^*_p$ and parameterized surface
$x(t,s)=\exp_p(t(u+sw))$. Let $\xi(1,s)$ be a cotangent lift of
the extremal $t\mapsto \exp_p(tv(s))$ at $t=1$. Since
$d(\exp_p)_uw=\dfrac{\partial x}{\partial s}(1,0)$ the statement
(ii) is reduced to ~\eqref{eq:main}.
\end{proof}

Let $c(t)$ be a $C^1$ piecewise curve in $M$ for $t\in (a,b)$, where $(a,b)$ is an interval in $\mathbb{R}$.
We remind that a curve $c(t)$ is called horizontal if $\dot{c}(t)\in S_x$ for any $t\in(a,b)$. A section $\xi(t)$
is called a cotangent lift of $c(t)$ if $\xi(t)\in T^*_{x(t)}$ and $g_x\xi=\dot{x}(t)$ for every $t$ where it is defined.
The notion of the natural parameter or arc length~\eqref{natpar} for $c(t)\colon (a,b)\to M$ can be reformulated as follows
$$L(c)=\int\limits_a^b\langle g_{c(t)}\xi(t),\xi(t)\rangle^{1/2}\,dt.$$

Let us focus for the moment on the case of sub-Lorentzian manifold. At each point $p\in M$ the distribution $S_p$ and the cotangent subbundle $T^*_p/S^{\bot}_p$ carry the structure of the Lorentz vector space and thus inherit the typical features of the Lorentz structure. Since the orthogonal complement $w^{\bot}$ to any timelike vector $w$ is spacelike then the vector space $S_p$ can be decomposed into the direct sum $\mathbb Rm\oplus w$. The same regards the cotangent vector space $T^*_p/S^{\bot}_p$. We define the future timecone in $S_p$ by
$$
C(\mathcal T(p))=\{X(p)\in S_p\colon Q_p(\mathcal T(p),X(p))<0\},\quad\text{where}\ \ \mathcal T\ \ \text{is the time orientation on}\ \ S_p.
$$
Analogously the future timecone in $T^*_p$ is
$$
C(\tau(p))=\{w\in T^*_p\colon \langle g_pw,\tau(p))<0\},\quad\text{where}\ \ \tau\ \ \text{is the time orientation on}\ \
T^*_p/S^{\bot}_p.$$ There is a consequence that vectors (covectors) $v,w$ are timelike if and only if $Q(v,w)<0$ ($\langle gv,w\rangle<0$). In vector spaces with positively definite metric the Schwarz inequality permits the definition of the angle $\theta$ between $v$ and $w$ as the unique number $0\leq \theta\leq \pi$. The analogues Lorentz result is as follows.
\begin{proposition}
Let $v$ and $w$ be timelike vectors in a Lorentz vector space equipped with the scalar product $\langle v,w\rangle$. Then
\begin{itemize}
\item [(1)]{$|\langle v,w\rangle|\geq |v||w|$, where $|v|=|\langle v,v\rangle|^{1/2}$, $|w|=|\langle w,w\rangle|^{1/2}$. The equality is possible if and only if $v$ and $w$ are collinear.}
\item [(2)]{There is a unique number $\vartheta>0$, called hyperbolic angle between $v$ and $w$, such that $$
\langle v,w\rangle=-|v||w|\cosh\vartheta.$$}
\end{itemize}
\end{proposition}
Consider a piecewise smooth timelike curve $c(t)$. The timelike means not only that every $\dot c(t)$ is timelike, but that at each break $t_i$ of $c$ $$Q_{c(t_i)}(\dot c(t_i^-),\dot c(t_i^+)<0.$$ Here the first vector derives from $c$ on the  interval $[t_{i-1},t_i]$ before break, and the second from the interval after break $[t_{i},t_{i+1}]$. Thus $\dot c$ does not switch timecone at a break. Similarly we require that a piecewise smooth causal curve does not switch causal cones at a break.

\begin{lemma}
Let $p$ be a point at Lorentz manifold $M$. Suppose that $\gamma\colon [0,b]\in T_p^*$ is a piecewise smooth curve starting at the origin such that $\alpha=exp_p\circ\gamma$ is timelike. Then $\gamma$ remains in a single timecone of $T_p^*$.
\end{lemma}

\begin{proof}
We consider two cases. The first one is related with the smooth curve $\gamma$ and the second case will be general. Thus, we assume that $\gamma(t)$, and hence $\alpha(t)$ are smooth in all the domain of definition. Since $g_p(\dot\gamma(0),\dot\gamma(0))=\langle g_p\gamma(0),\gamma(0)\rangle=Q_p(\dot\alpha(0),\dot\alpha(0))<0$, then the curve $\gamma(t)$ is in the same timecone for $t\in(0,\varepsilon)$, where $\varepsilon$ is sufficiently small. We also conclude that $\dot \gamma(t)$ maintains in the same timecone with $t\in(0,\varepsilon)$ for sufficiently small $\varepsilon>0$. Let us denote by $r_{\gamma(t)}$ the radial tangent vector in $T_{\gamma(t)}(T^*_p)$ corresponding to the timelike covector $\gamma(t)$. The vector $r_{\gamma(t)}$ is timelike and therefore, $\langle g_p \dot\gamma(t), r_{\gamma(t)}\rangle$ is negative for $t\in(0,\varepsilon)$. We calculate \begin{equation}\label{eq:4.13}\frac{d}{dt}\langle g_p\gamma(t),\gamma(t)\rangle=2\langle g_p\dot \gamma(t),\gamma(t)\rangle\end{equation} is negative for $t\in(0,\varepsilon)$. Since $\dot\alpha(t)=d(exp_p)_{\gamma(t)}\dot\gamma$ and it is in $S_{exp_p\gamma(t)}$ the last expression in~\eqref{eq:4.13} is equal to $2Q_{exp_p\gamma(t)}\big(\dot\alpha(t), d(exp_p)_{\gamma(t)}r_{\gamma(t)}\big)$ by Lemma~\ref{gauss}. We conclude that so long as $\gamma$ remains in timecone the radial vector $r_{\gamma}$ and the vector $d(exp_p)_{\gamma(t)}r_{\gamma(t)}$ remains timelike. Thus $Q_{exp_p\gamma(t)}\big(\dot\alpha(t), d(exp_p)_{\gamma(t)}r_{\gamma(t)}\big)$ hence $\langle g_p\dot \gamma(t),\gamma(t)\rangle$ and hence $\frac{d}{dt}\langle g_p\gamma(t),\gamma(t)\rangle$ remains negative. But $\gamma$ can leave the timecone only by reaching null-cone or the origin. In any of these cases $\langle g_p\gamma(t),\gamma(t)\rangle=0$. Thus $\gamma$ must remain in the same timecone.

Now suppose that $\gamma$ and hence $\alpha$ is piecewise smooth. We know from the first part of the proof that on its first smooth segment $\gamma$ stays in the same timecone and therefore in the first break $$\langle g_{\alpha(t_0)}\dot\gamma(t_1^-),r_{\beta(t_1)}\rangle<0.$$ Hence by Lemma~\ref{gauss} $$Q_{exp_p\gamma(t_1)}\big(\dot\alpha(t_1^-), d(exp_p)_{\gamma(t_1)}r_{\gamma(t_1)}\big)<0.$$ The additional condition on $\alpha$ at breaks keeps $\dot\alpha(t_1^+)$ in the same timecone, namely at $r_1=d(exp_p)_{\gamma(t_1)}r_{\gamma(t_1)}$. So, again by Lemma~\ref{gauss} $\langle g_{\alpha(t_1)}\dot\gamma(t_1^+),r_{\beta(t_1)}\rangle<0$ and therefore it follows as above that $\frac{d}{dt}\langle g\gamma,\gamma\rangle$ can not change signs at breaks. Hence the argument  for the smooth case remains valid.
\end{proof}

Minor changes in this proof show that the lemma remains true if the words {\it timelike} and {\it timecone} are replaced by {\it causal} and {\it causal cone}.

\begin{lemma}
Let $U$ be a normal neighborhood of $p$ in a Lorentz manifold. If
there exists a piecewise timelike curve $\alpha$ in $U$ from $p$
to $q$, then the  segment $\sigma$ of the extremal from $p$ to $q$
is the unique longest timelike curve in $U$ from $p$ to $q$.
\end{lemma}
\begin{proof} We understand the uniqueness as the uniqueness up to a monotone reparametrization and can suppose that $\alpha$ is parameterized by arc length.
If $\alpha\colon [0,b]\to U$ is a timelike curve in $U$ from
$p=\alpha(0)$ to $q=\alpha(b)$, then from the lemma above the
lifting $w(t)\colon [0,b]\to T_p^*$, $\alpha(t)=\exp_p\circ w(t)$
remains in a single timecone in $T^*_p$. The section $w(t)$ is the
timelike vector for any $t\in[0,1]$ and therefore define a unite
timelike section $u(t)=|\langle g_pw(t),w(t) \rangle
|^{-1/2}w(t)$. Since $\langle g_p u(t),\dot u(t)\rangle=0$ the
vector field $\dot u(t)$ is spacelike. Let us write $r(t)=|\langle
g_pw(t),w(t)\rangle|^{1/2}$, then $w(t)=r(t)u(t)$. We calculate
$$\dot\alpha(t)=\frac{d}{dt}exp_pw(t)=d(exp_p)_{w(t)}\dot
r(t)u(t)+d(exp_p)_{w(t)}r(t)\dot u(t).$$ Let us denote by $\xi(t)$
the cotangent lift of the extremal $\gamma:s\mapsto \exp_p(su(t))$
at $s=r(t)$. Since $u(t)$ is timelike the same does $\xi(t)$.
Since $\dot \alpha$ is horizontal, then $d(exp_p)_{w(t)}r(t)\dot
u(t)\in S_{\alpha(t)}$. It means that there is spacelike section
$\eta(t)$ such that $d(exp_p)_{w(t)}r(t)\dot
u(t)=g_{\alpha(t)}\eta(t)$ almost everywhere and orthogonal to the
section $\xi$ by Lemma~\ref{gauss}. Then the length of timelike
vector $\dot\alpha$ is given by
\begin{eqnarray*}\Big(-Q_{\alpha(t)}(\dot\alpha,\dot\alpha)\Big)^{1/2} & = & \Big(-\langle \dot r(t)\xi(t)+\eta,\dot r(t)g_{\alpha(t)}\xi(t)+g_{\alpha(t)}\eta(t)\rangle \Big)^{1/2}\\
& = & \Big(|\dot r(t)|^2-\langle g_{\alpha(t)}\eta(t),\eta(t)
\rangle\Big)^{1/2}\leq |\dot r(t)|.
\end{eqnarray*}
Therefore
$$L(\alpha)=\int\limits_0^b|Q_{\alpha(t)}(\dot\alpha,\dot\alpha)|^{1/2}\,dt\leq\int\limits_0^b|\dot r(t)|\,dt=|r(b)|=L(\sigma).$$

The equality holds if and only if $\dot r(t)$ is monotone and $\langle g_{\alpha(t)}\eta(t),\eta(t)
\rangle=0$. In this case the velocity of $\alpha$ satisfies the equation $\dot\alpha(t)=\dot r(t)d(exp_p)_{w(t)}u(t)=\dot r(t)g_{\alpha(t)}\xi(\alpha(t))$. From the other hand the extremal $\gamma(s)=\exp_p(r(t)u(t_0))$, $s=r(t)$, satisfies the equation $\dot\gamma(s)=d(exp_p)_{r(t)u(t_0)}u(t_0)=\dot sg_{\gamma(s)}\xi(\gamma(s))$. Since $\alpha(t)$ and $\gamma(s)$ satisfy the same equation and have the same initial point, we conclude that $\alpha$ is a reparameterization of the extremal $\gamma$.
\end{proof}

We have noticed that a general piecewise smooth horizontal curve does not have a unique cotangent lift. If the curve is an extremal then there is a special cotangent lift, the one that satisfies the Hamilton-Jacobi equation. In the case of the two step bracket generating distributions it is possible to find a canonical cotangent lift. The condition for this is formulated in the following lemma.

\begin{lemma} Assume the strong bracket generating hypothesis.
    Let $x(t)$ be any Lipshitz horizontal curve. Then there exists such a cotangent lift $(x(t),\xi(t))$ that
    a cotangent vector $\omega_j=\dot{\xi}_j+\dfrac{1}{2}\dfrac{\partial g^{pq}}{\partial x^j}\xi_p\xi_q$
    is orthogonal to $\Gamma^j(\xi,v(x))$ for any $v\in S^{\bot}_x\; at\; a.\,e.\; t$ so that pairing $\langle \omega_j,\,\Gamma^j(\xi,v(x))\rangle=0$. This cotangent lift does not depend on the coordinate system and is called a canonical cotangent lift.
\end{lemma}
\begin{proof}
    Let $(x(t),\eta(t))$ be any cotangent lift and let $v^{(1)},\ldots, v^{(n-m)}$ be a basis of sections of null-space $S^{\bot}$ over a neighborhood of the curve. Since $\xi$ belongs in general to $T^*_x/S^{\bot}_x\times S^{\bot}_x$ then we can write
   \begin{gather*}
       \xi(t)=\eta(t)+v(t)=\eta(t)+\sum\limits_{k=1}^{n-m}a_k(t)v^{(k)}(t).
   \end{gather*}
Then \begin{gather*}
         \dot{\xi}_j+\dfrac{1}{2}\dfrac{\partial g^{pq}(x)}{\partial x^j}\xi_p\xi_q=\dot{\eta}_j+\dot{v}_j+\dfrac{1}{2}\dfrac{\partial g^{pq}}
         {\partial x^j}(\eta_p+v_p)(\eta_q+v_q)\\
         =\dot{\eta}_j+\dfrac{1}{2}\dfrac{\partial g^{pq}}{\partial x^j}\eta_p\eta_q+
         \biggl(\dot{v}_j+\dfrac{1}{2}\dfrac{\partial g^{pq}}{\partial x^j}\eta_pv_q\biggr).
     \end{gather*}
Recall that for $w\in S^{\bot}_x$
\begin{gather*}
    \Gamma^k(\xi,w)=\dfrac{1}{2}\biggl(g^{jp}g^{kq}\xi_p\dfrac{\partial w_q}{\partial x_j}+g^{jk}\dfrac{\partial g^{pq}}{\partial x^j}w_q\xi_p\biggr)
    =\dfrac{1}{2}g^{jk}\bigl(\dot{w}_j+\dfrac{\partial g^{pq}}{\partial x^j}\xi_pw_q\bigr).
\end{gather*}
Here we used that $g^{pq}\xi_p\dfrac{\partial w_j}{\partial x_k}=\dot{x}_q\dfrac{\partial w_j}{\partial x_q}=\dot{w}_j$. Notice that $\Gamma^k(\xi,w)=\Gamma^k(\eta,w)$.
We had already shown that $\Gamma^j(\xi,w)$ transforms as a tangent vector. Show now that \\$\dot{\xi}_j+\dfrac{1}{2}\dfrac{\partial g^{pq}(x)}{\partial x^j}\xi_p\xi_q$ transforms as a cotangent vector. Consider the transformation laws
$$\xi_k=\dfrac{\partial y^j}{\partial x^k}\widetilde{\xi}_j\quad \mbox{and}\quad \widetilde{g^{pq}}=g^{rs}\dfrac{\partial y^p}{\partial x^r}
 \dfrac{\partial y^q}{\partial x^s}.$$ Then
\begin{gather*}
    \dot{\widetilde{\xi}}_j+\dfrac{1}{2}\dfrac{\partial \widetilde{g}^{pq}}{\partial y^j}\widetilde{\xi}_p\widetilde{\xi}_q=
    \widetilde{g}^{pq}\widetilde{\xi}_p\dfrac{\partial \widetilde{\xi}_j}{\partial y^q}
    +\dfrac{1}{2}\dfrac{\partial \widetilde{g}^{pq}}{\partial y^j}\widetilde{\xi}_p\widetilde{\xi}_q\\
    =g^{rs}\dfrac{\partial y^p}{\partial x^r}\dfrac{\partial y^q}{\partial x^s}\dfrac{\partial x^l}{\partial y^p}\xi_l
    \dfrac{\partial }{\partial x^m}\biggl(\dfrac{\partial x^k}{\partial y^j}\xi_k\biggr)\dfrac{\partial x^m}{\partial y^q}
    +\dfrac{1}{2}\dfrac{\partial }{\partial x^l}\biggl(g^{rs}\dfrac{\partial y^p}{\partial x^r}\dfrac{\partial y^q}{\partial x^s}\biggr)
    \dfrac{\partial x^l}{\partial y^j}\dfrac{\partial x^k}{\partial y^p}\xi_k\dfrac{\partial x^n}{\partial y^q}\xi_n\\
    =g^{rs}\dfrac{\partial y^p}{\partial x^r}\dfrac{\partial y^q}{\partial x^s}\dfrac{\partial x^l}{\partial y^p}
    \dfrac{\partial^2 x^k}{\partial x^m\partial y^j}\dfrac{\partial x^m}{\partial y^q}\xi_l\xi_k
    +g^{rs}\dfrac{\partial y^p}{\partial x^r}\dfrac{\partial y^q}{\partial x^s}\dfrac{\partial x^l}{\partial y^p}
    \dfrac{\partial x^k}{\partial y^j}\dfrac{\partial \xi_k}{\partial x^m}\dfrac{\partial x^m}{\partial y^q}\xi_l\\
    +\dfrac{1}{2}\dfrac{\partial g^{rs}}{\partial x^l}\dfrac{\partial y^p}{\partial x^r} \dfrac{\partial y^q}{\partial x^s}
    \dfrac{\partial x^l}{\partial y^j}\dfrac{\partial x^k}{\partial y^p}\dfrac{\partial x^n}{\partial y^q}\xi_k\xi_n
    +\dfrac{1}{2}g^{rs}\biggl(\dfrac{\partial^2y^p}{\partial x^l\partial x^r}\dfrac{\partial y^q}{\partial x^s}
    +\dfrac{\partial y^p}{\partial x^r}\dfrac{\partial^2y^q}{\partial x^l\partial x^s}\biggr)
    \dfrac{\partial x^l}{\partial y^j}\dfrac{\partial x^k}{\partial y^p}\dfrac{\partial x^n}{\partial y^q}\xi_k\xi_n.
\end{gather*}
The second and the third terms here are equal to $g^{rs}\dfrac{\partial x^k}{\partial y^j}\dfrac{\partial \xi_k}{\partial x^s}\xi_r$ and $\dfrac{1}{2}
\dfrac{\partial g^{rs}}{\partial x^k}\dfrac{\partial x^k}{\partial y^j}\xi_r\xi_s$ respectively, which gives in the whole the transformation of a covariant vector:
\begin{gather*}
   \dfrac{\partial x^k}{\partial y^j}\biggl(g^{rs}\xi_r\dfrac{\partial \xi_k}{\partial x^s}+\dfrac{1}{2}\dfrac{\partial g^{rs}}{\partial  x^k}\xi_r\xi_s\biggr)=\dfrac{\partial x^k}{\partial y^j}\biggl(\dot{\xi}_k+\dfrac{1}{2}\dfrac{\partial g^{rs}}{\partial x^k}\xi_r\xi_s\biggr).
\end{gather*}
The rest of the terms give in sum
\begin{gather*}
    g^{rs}\dfrac{\partial^2x^k}{\partial x^s\partial y^j}\xi_r\xi_k+\dfrac{1}{2}g^{rs}\dfrac{\partial^2y^p}{\partial x^l\partial x^s}
    \dfrac{\partial x^l}{\partial y^j}\dfrac{\partial x^k}{\partial y^p}\xi_k\xi_r+\dfrac{1}{2}g^{rs}\dfrac{\partial^2y^p}{\partial x^l\partial x^s}
    \dfrac{\partial x^l}{\partial y^j}\dfrac{\partial x^k}{\partial y^p}\xi_r\xi_k\\
    =g^{rs}\dfrac{\partial^2x^k}{\partial x^s\partial y^j}\xi_r\xi_k+g^{rs}\dfrac{\partial^2y^p}{\partial x^l\partial x^s}
    \dfrac{\partial x^l}{\partial y^j}\dfrac{\partial x^k}{\partial y^p}\xi_r\xi_k=g^{rs}\xi_r\xi_k\biggl(\dfrac{\partial^2x^k}{\partial x^s\partial y^j}+\dfrac{\partial^2y^p}{\partial x^l\partial x^s}\dfrac{\partial x^l}{\partial y^j}\dfrac{\partial x^k}{\partial y^p}\biggr)=0
\end{gather*}
since $\dfrac{\partial }{\partial y^j}\delta^k_s=0$, where $\delta^k_s$ is a Kronecker symbol.

Now we see that the orthogonality condition is of the form
\begin{gather*}
   \biggl[\biggl(\dot{\eta}_j+\dfrac{1}{2}\dfrac{\partial g^{pq}}{\partial x^j}\eta_p\eta_q\biggr)+
   \biggl(\dot{v}_j+\dfrac{1}{2}\dfrac{\partial g^{pq}}{\partial x^j}\eta_pv_q\biggr)\biggr]\cdot\Gamma^k(\eta,w)=0.
\end{gather*}
As $\Gamma(\xi,\cdot)$ is injective the converse matrix $(\Gamma^k(\eta,w))^{-1}$ exists. Therefore the linear system of $n-m$ equations in $n-m$ variables $a_k(t)$
$$\sum\limits_{k=1}^{n-m}(\dot{a}_{k} v^{(k)}_j+a_{k}\dot{v}^{(k)}_j)+\dfrac{\partial g^{pq}}{\partial x^j}\eta_p\sum\limits_{k=1}^{n-m}a_{k}v^{(k)}_q
=(\Gamma^k(\eta,w))^{-1}(\dot{\eta}_j+\dfrac{1}{2}\dfrac{\partial g^{pq}}{\partial x^j}\eta_p\eta_q)$$
is uniquely solvable.
\end{proof}

\section{Differential of the exponential map}\label{sec:5}

As it was mentioned, the exponential mapping $\exp_p$ is not a
diffeomorphism at the origin, but as in the case of sub-Riemannian
geometry there is a hope that it is a local diffeomorphism at some
points. The main result can be stated that the exponential map
$\exp_p(u)$ is a local diffeomorphism if $u$ is neither a null
vector no an annihilator. We consider only the case of 2-step
bracket generating distribution. First, let us set out the Taylor
expansion for $k$-th component of $\exp_p(u)$, where $p$ is fixed
at the origin of the coordinates and $u\in T^*_pM$:
\begin{equation}\label{eq:5.1}
    \exp_p(u)^k=\sum\limits_{r=1}^N\dfrac{1}{r!}\gamma_{(r)}^{kp_1\ldots p_r}u_{p_1}\ldots u_{p_r}+O(|u|^{N+1}),
\end{equation}
where $\gamma_{(r)}^{kp_1\ldots p_r}$ is symmetric in indexes
$p_1,\ldots p_r$ and will be computed later, $|u|$ is any
Euclidean norm on $T^*_pM$. Notice that $\exp_p(tu)=x(t)$, where
$(x(t),\xi(t))$ --- solution of the system \eqref{eq:ham} with
$x(0)=0$, $\xi(0)=u$. Then at the origin
$$\gamma_{(r)}^{kp_1\ldots p_r}u_{p_1}\ldots u_{p_r}=\left(\dfrac{d}{dt}\right)^rx^k(0).$$
We count for some value of $t$
\begin{gather*}
    \left(\dfrac{d}{dt}\right)^{r+1}\!\!x^k(t)=\dfrac{d}{dt}\left(\gamma_{(r)}^{kp_1\ldots p_r}(x(t))\xi_{p_1}(t)\ldots\xi_{p_r}(t)\right)\notag\\
    =\dfrac{\partial\gamma_{(r)}^{kp_1\ldots p_r}}{\partial x^q}(x(t))\cdot\dot{x}^q(t)\cdot\xi_{p_1}(t)\ldots \xi_{p_r}(t)+r\cdot\gamma_{(r)}^{kp_1\ldots p_r}
    \cdot\dfrac{\partial \xi_{p_i}}{\partial t}(t)\cdot\xi_{p_1}(t)\ldots \widehat{\xi}_{p_i}(t)\ldots \xi_{p_r}(t),
\end{gather*} where $\widehat{\xi}_{p_i}(t)$ denotes the absence of $\xi_{p_i}(t)$.
Now, using \eqref{eq:ham} and changing indexes, we get
\begin{eqnarray*}\left(\dfrac{d}{dt}\right)^{r+1}\!\!x^k(t)
& = & \Big( \dfrac{\partial\gamma_{(r)}^{kp_1\ldots p_r}}{\partial x^q}(x(t))\cdot g^{qp_{r+1}}(x(t))
\\
& - & \dfrac{r}{2}\cdot\gamma_{(r)}^{kp_1\ldots p_{r-1}q}
    \cdot\dfrac{\partial g^{p_rp_{r+1}}}{\partial x^q}(x(t)) \Big) \cdot\xi_{p_1}(t)\ldots \xi_{p_{r+1}}(t).
\end{eqnarray*}
Therefore,
\begin{equation}\label{eq:5.2}
    \gamma_{r+1}^{kp_1\ldots p_{r+1}}(x)=\sym(p_1,\ldots,p_{r+1})\cdot\biggl(g^{qp_{r+1}}(x)\dfrac{\partial\gamma_{(r)}^{kp_1\ldots p_r}}{\partial x^q}(x) -\dfrac{r}{2}\gamma_{(r)}^{kp_1\ldots p_{r-1}q}(x)
    \dfrac{\partial g^{p_rp_{r+1}}}{\partial x^q}(x)\biggr),
\end{equation}
here $sym(p_1,\ldots,p_{r+1})$ means that we symmetrize the indexes $p_1,\ldots,p_{r+1}$.
Setting $r=0$ in the previous formula we get
\begin{equation}\label{eq:5.3}
    \gamma_{(1)}^{kp}=g^{qp}\dfrac{\partial\gamma^k_{(0)}}{\partial x^q}=g^{kp},
\end{equation}
since $\gamma^k_{(0)}=x^k(0)$ and $\dfrac{\partial x^k}{\partial x^q}$ equals to 1 if and only if $k=q$ and zero otherwise.

Analogously, observe that for $r=1$ in \eqref{eq:5.2}
\begin{eqnarray}\label{eq:5.4}
\gamma_{(2)}^{kp_1p_2} & = & \sym(p_1,p_2)\cdot\biggl(g^{qp_2}\dfrac{\partial\gamma_{(1)}^{kp_1}}{\partial x^q}-\dfrac{1}{2}\gamma_{(1)}^{kq}
    \dfrac{\partial g^{p_1p_2}}{\partial x^q}\biggr)\\
& = & \sym(p_1,p_2)\cdot\biggl(g^{qp_2}\dfrac{\partial g^{kp_1}}{\partial x^q}-\dfrac{1}{2}g^{kq}\dfrac{\partial g^{p_1p_2}}{\partial x^q}\biggr)=-\Gamma^{kp_1p_2}.\nonumber
\end{eqnarray}

It is rather hard to calculate a general term, but it will be
sufficient for us to look into the view of $\gamma_{(3)}$.

Now \eqref{eq:5.1} receives the following form
\begin{gather*}
    \exp_p(u)^k=\gamma_{(1)}^{kp_1}u_{p_1}+\sum\limits_{r=2}^N\dfrac{1}{r!}\gamma_{(r)}^{kp_1\ldots p_r}u_{p_1}\ldots u_{p_r}+O(|u|^N)
\end{gather*}
and differentiating it, we obtain
\begin{equation}\label{eq:5.5}
    d\exp_p(u)^k=\gamma_{(1)}^{kp_1}+\sum\limits_{r=2}^N\dfrac{1}{(r-1)!}\gamma_{(r)}^{kp_1\ldots p_r}u_{p_2}\ldots u_{p_r}+O(|u|^N)
\end{equation}
\begin{gather*}
    =g^{kj}(0)+\sum\limits_{r=2}^N\dfrac{1}{(r-1)!}\,\gamma_{(r)}^{kjp_2\ldots p_r}\,u_{p_2}\ldots u_{p_r}+O(|u|^N)\notag
\end{gather*}
More precisely,
\begin{gather*}
 d\exp_p(u)^{kj}=g^{kj}(0)+\gamma_{(2)}^{kjp_2}u_{p_2}+\dfrac{1}{2}\gamma_{(3)}^{kjp_2p_3}u_{p_2}u_{p_3}+O(|u|^3)\\
 =g^{kj}(0)-\Gamma^{kjp}u_p+\dfrac{1}{2}\gamma_{(3)}^{kjpq}u_pu_q+O(|u|^3).
\end{gather*}

Since we assumed 2-step bracket generating hypothesis, choose
coordinates near $p$ so that $p$ is an origin and
\begin{gather*}
    g^{jk}(0)=
\left(\begin{matrix}
\varepsilon_jI^{jk} & 0 \\
0 & 0
\end{matrix}\right),
\end{gather*}
where $I^{jk}$ is a $m\times m$ unit matrix and $\varepsilon_jI^{jk}$ is a $m\times m$ matrix with $\nu$ negative unities on the diagonal and $m-\nu$ positive unities, which can be also written as follows:
$g^{jk}(0)=\varepsilon_j\delta^j_{k}$ , where $\delta^j_{k}$ is a Kronecker symbol and
$$\varepsilon_j=\begin{cases}-1, \;\;\mbox{if}\;\; 1\leqslant j\leqslant\nu,\\ 1,\;\;\mbox{if}\;\; \nu<j\leqslant m,\\
 0,\;\;\mbox{if}\;\; m<j\leqslant n.\end{cases}$$

Denote with $a,\,b$ the indexes responsible for elements standing
in rows or columns with numbers $1,\ldots,m$, and $\alpha,\,\beta$
--- for $m+1,\ldots,n$ respectively. Then $d\exp_p(u)$ is a
$n\times n$ matrix of the following form
\begin{equation*}
    W^{kj}=\left(\begin{matrix}
A^{ab} & B^{a\beta} \\
C^{\alpha b} & D^{\alpha\beta}
\end{matrix}\right)
\end{equation*}
with
\begin{eqnarray}\label{eq:ABCD}
   &A^{ab}=\varepsilon_aI^{ab}+O(|u|),\nonumber\\
   &B^{a\beta}=-\Gamma^{a\beta p}u_p+O(|u|^2),\\
   &C^{\alpha b}=-\Gamma^{\alpha b p}u_p+O(|u|^2),\nonumber\\
   &D^{\alpha\beta}=\dfrac{1}{2}\gamma_{(3)}^{\alpha\beta pq}u_pu_q+O(|u|^3).\nonumber
\end{eqnarray}
Since $\gamma^{\alpha\beta p}_{(2)}=0$ due to the special choice of $g^{kj}$, there are no terms of order 2 in $D^{\alpha\beta}$. The following proposition is an easy computation on determinant.
\begin{lemma}\label{lem:5.1}
$\det W(u)=\det \widetilde{W}(u)+O(|u|^{2(n-m)+1})$, where $\widetilde{WW}(u)$ is obtained from $W$ by discarding the error terms containing $O(|u|^i)$, $i=1,\,2,\,3$ and $\det \widetilde{W}(u)$ is homogeneous of degree $2(n-m)$ in $u$.
\end{lemma} To estimate the determinant of $\widetilde{W}(u)$ we need some more calculations.
From \eqref{eq:5.2} and \eqref{eq:5.4} we get
\begin{eqnarray*}
\gamma_{(3)}^{\alpha\beta pq}(x) & = & \sym(\beta,p,q)\cdot\biggl(g^{jq}(x)\dfrac{\partial \gamma_{(2)}^{\alpha\beta p}}{\partial x^j}(x)
    -\gamma_{(2)}^{\alpha\beta j}(x)\dfrac{\partial g^{pq}}{\partial x^j}(x)\biggr)\\
& = & \sym(\beta,p,q)\cdot\biggl(-g^{jq}(x)\dfrac{\partial \Gamma^{\alpha\beta p}}{\partial x^j}(x)+\Gamma^{\alpha\beta j}(x)\dfrac{\partial g^{pq}}{\partial x^j}(x)\biggr)\\
& = & \dfrac{1}{3}\biggl(\Gamma^{\alpha\beta j}(x)\dfrac{\partial g^{pq}}{\partial x^j}(x)+\Gamma^{\alpha pj}(x)\dfrac{\partial g^{\beta q}}{
    \partial x^j}(x)+\Gamma^{\alpha qj}\dfrac{\partial g^{p\beta}}{\partial x^j}(x)\biggr.\\
\biggl.& - & g^{jq}(x)\dfrac{\partial \Gamma^{\alpha\beta p}}{\partial x^j}(x)-g^{jp}(x)\dfrac{\partial \Gamma^{\alpha\beta q}}{\partial x^j}(x)-
    g^{j\beta}(x)\dfrac{\partial \Gamma^{\alpha pq}}{\partial x^j}(x)\biggr).
\end{eqnarray*}
Setting here $x=0$, we get that the first and the last terms in the last sum are zero, because $g^{jk}(0)=0$ for $j,k>m$. Hence, for $p,\,q\leqslant m$
\begin{equation}\label{eq:5.8}
    \gamma_{(3)}^{\alpha\beta ab}=\dfrac{1}{3}\biggl(\Gamma^{\alpha aj}\dfrac{\partial g^{\beta b}}{
    \partial x^j}+\Gamma^{\alpha bj}\dfrac{\partial g^{a\beta}}{\partial x^j}
    -\varepsilon_b\dfrac{\partial \Gamma^{\alpha\beta a}}{\partial x^b}-\varepsilon_a\dfrac{\partial \Gamma^{\alpha\beta b}}{\partial x^a}\biggr)
\end{equation}
and for $p,\,q>m$ $\gamma_{(3)}^{\alpha\beta pq}=0$ since $g^{jk}(0)=0$ for $j,k>m$.
Let us calculate the involved terms in \eqref{eq:5.8}.
\begin{gather*}
    \left.\Gamma^{\alpha aj}\right|_{x=0}=\left.\dfrac{1}{2}\left(g^{\alpha k}\dfrac{\partial g^{aj}}{\partial x^k}-g^{a k}\dfrac{\partial g^{\alpha j}}{\partial x^k}
    -g^{j k}\dfrac{\partial g^{\alpha a}}{\partial x^k}\right)\right|_{x=0}=-\dfrac{1}{2}\left(\varepsilon_a \dfrac{\partial g^{\alpha j}}{\partial x^a}+\varepsilon_j\dfrac{\partial g^{\alpha a}}{\partial x^j}\right),
\end{gather*}
\begin{eqnarray*}
\Gamma^{\alpha \beta a}\Big|_{x=0} & = & \dfrac{1}{2}\left(g^{\alpha k}\dfrac{\partial g^{\beta a}}{\partial x^k}-g^{\beta k}
     \dfrac{\partial g^{\alpha a}}{\partial x^k}
    -g^{a k}\dfrac{\partial g^{\alpha \beta}}{\partial x^k}\right)\Big|_{x=0}=-\dfrac{1}{2}\varepsilon_a
    \dfrac{\partial g^{\alpha \beta}}{\partial x^a},\\
\dfrac{\partial \Gamma^{\alpha\beta a}}{\partial x^b}\Big|_{x=0} & = & \dfrac{1}{2}\Big(\dfrac{\partial g^{\alpha k}}{\partial x^b}
    \dfrac{\partial g^{\beta a}}{\partial x^k}-g^{\alpha k}\dfrac{\partial^2 g^{\beta a}}{\partial x^b \partial x^k}-
    \dfrac{\partial g^{\beta k}}{\partial x^b}\dfrac{\partial g^{\alpha a}}{\partial x^k}\\
& + & g^{\beta k}
    \dfrac{\partial^2 g^{\alpha a}}{\partial x^b x^k}-\dfrac{\partial g^{a k}}{\partial x^b}\dfrac{\partial g^{\alpha \beta}}{\partial x^k}
    -g^{ak}\dfrac{\partial^2 g^{\alpha \beta}}{\partial x^bx^k}\Big)\Big|_{x=0}\\
& = & \dfrac{1}{2}\left(\dfrac{\partial g^{\alpha k}}{\partial x^b}
    \dfrac{\partial g^{\beta a}}{\partial x^k}-
    \dfrac{\partial g^{\beta k}}{\partial x^b}\dfrac{\partial g^{\alpha a}}{\partial x^k}
    -\varepsilon_a\dfrac{\partial^2 g^{\alpha \beta}}{\partial x^a\partial x^b}\right)
\end{eqnarray*}
owing to $g^{\alpha k}=g^{\beta k}=0$ and $\dfrac{\partial
g^{\alpha\beta}}{\partial x^k}=0$ by Lemma \ref{lemma:2.1}. Now we
simplify the form of $\dfrac{\partial^2 g^{\alpha \beta}}{\partial
x^a\partial x^b}$. Take a null-section $v(x)$, then
$$g^{jk}\dfrac{\partial v_k}{\partial x^p}=-\dfrac{\partial g^{jk}}{\partial x^p}v_k.$$
Thus, making use of Lemma~\ref{lemma:2.1} and differentiating both
parts, we obtain
\begin{gather*}
     \dfrac{\partial}{\partial x^q}\biggl(g^{jk}\dfrac{\partial v_k}{\partial x^p}\biggr)=\dfrac{\partial g^{jk}}{\partial x^q}\dfrac{\partial v_k}{\partial x^p}+g^{jk}\dfrac{\partial^2v_k}{\partial x^p\partial x^q},\\
     \dfrac{\partial}{\partial x^q}\biggl(\dfrac{\partial g^{jk}}{\partial x^p}v_k\biggr)=\dfrac{\partial^2 g^{jk}}{\partial x^q\partial x^p}v_k
     +\dfrac{\partial g^{jk}}{\partial x^p}\dfrac{\partial v_k}{\partial x^q}.
\end{gather*}
From here
\begin{gather*}
     \dfrac{\partial^2 g^{jk}(x)}{\partial x^q\partial x^p}v_k(x)=-\dfrac{\partial g^{jk}(x)}{\partial x^p}\dfrac{\partial v_k(x)}{\partial x^q}-\dfrac{\partial g^{jk}(x)}{\partial x^q}\dfrac{\partial v_k(x)}{\partial x^p}-g^{jk}(x)\dfrac{\partial^2v_k(x)}{\partial x^p\partial x^q}.
\end{gather*}
Taking inner product with another null-section, we get
\begin{gather*}
     \dfrac{\partial^2 g^{jk}(x)}{\partial x^q\partial x^p}v_k(x)w_j(x)=-\dfrac{\partial g^{jk}(x)}{\partial x^p}\dfrac{\partial v_k(x)}{\partial x^q}w_j(x)-\dfrac{\partial g^{jk}(x)}{\partial x^q}\dfrac{\partial v_k(x)}{\partial x^p}w_j(x)
\end{gather*}
since $g^{jk}(x)\dfrac{\partial^2v_k(x)}{\partial x^p\partial
x^q}w_j(x)=0$ by virtue of Lemma~\ref{lemma:2.1}.

Set $x=0$ and, since $g^{lk}(0)\dfrac{\partial v_k}{\partial x^r}(0)=\varepsilon_l\dfrac{\partial v_l}{\partial x^r}(0)=-\varepsilon_l\dfrac{\partial g^{lk}}{\partial x^r}(0)v_k$, then
\begin{eqnarray*}
\dfrac{\partial^2 g^{jk}}{\partial x^a\partial x^b}v_{k}w_{j} & = & -\dfrac{\partial g^{jk}}{\partial x^a}\dfrac{\partial v_{k}}{\partial x^b}w_{j}-\dfrac{\partial g^{jk}}{\partial x^b}\dfrac{\partial v_{k}}{\partial x^a}w_{j}\\
& = & \varepsilon_k\dfrac{\partial g^{jk}}{\partial x^a}
\dfrac{\partial g^{km}}{\partial x^b}v_mw_{j}+\varepsilon_k
\dfrac{\partial g^{jk}}{\partial x^b}\dfrac{\partial
g^{km}}{\partial x^a}v_mw_{j}.
\end{eqnarray*}
Take $v_k(0)=\delta^k_{\beta}$, $w_j(0)=\delta^j_{\alpha}$ and get
$$\dfrac{\partial^2 g^{\alpha\beta}}{\partial x^a\partial x^b}=\varepsilon_k\dfrac{\partial g^{\alpha k}}{\partial x^a}\dfrac{\partial g^{k\beta}}{\partial x^b}+\varepsilon_k\dfrac{\partial g^{\alpha k}}{\partial x^b} \dfrac{\partial g^{k\beta }}{\partial x^a}. $$
Therefore,
\begin{gather*}
     \dfrac{\partial\Gamma^{\alpha\beta a}}{\partial x^b}=\dfrac{1}{2}\left(\dfrac{\partial g^{\alpha j}}{\partial x^b}\dfrac{\partial g^{\beta a }}{\partial x^j}-\dfrac{\partial g^{\beta j}}{\partial x^b}\dfrac{\partial g^{\alpha a}}{\partial x^j}-\varepsilon_a\varepsilon_j\dfrac{\partial g^{\alpha j}}{\partial x^a}\dfrac{\partial g^{\beta j}}{\partial x^b}-\varepsilon_a\varepsilon_j\dfrac{\partial g^{\alpha j}}{\partial x^b}\dfrac{\partial g^{\beta j}}{\partial x^a}\right).
\end{gather*}
Substituting calculated terms in \eqref{eq:5.8}
\begin{eqnarray*}
\gamma_{(3)}^{\alpha\beta ab} & = & \dfrac{1}{6}\biggl[-\biggl(\varepsilon_a\dfrac{\partial g^{\alpha j}}{\partial x^a}+\varepsilon_j\dfrac{\partial g^{\alpha a}}{\partial x^j}\biggr)\dfrac{\partial g^{\beta b}}{\partial x^j}-\biggl(\varepsilon_b\dfrac{\partial g^{\alpha j}}{\partial x^b}+\varepsilon_j\dfrac{\partial g^{\alpha b}}{\partial x^j}\biggr)\dfrac{\partial g^{\beta a}}{\partial x^j}\\
& - & \varepsilon_b\biggl(\dfrac{\partial g^{\alpha j}}{\partial x^b}\dfrac{\partial g^{\beta a}}{\partial x^j}-\dfrac{\partial g^{\beta j}}{\partial x^b}\dfrac{\partial g^{\alpha a}}{\partial x^j}
    -\varepsilon_a\varepsilon_j\dfrac{\partial g^{\alpha j}}{\partial x^a}\dfrac{\partial g^{\beta j}}{\partial x^b}-\varepsilon_a\varepsilon_j\dfrac{\partial g^{\alpha j}}{\partial x^b}\dfrac{\partial g^{\beta j}}{\partial x^a}\biggr)\\
& - & \varepsilon_a\biggl(\dfrac{\partial g^{\alpha j}}{\partial x^a}\dfrac{\partial g^{\beta b}}{\partial x^j}-\dfrac{\partial g^{\beta j}}{\partial x^a}\dfrac{\partial g^{\alpha b}}{\partial x^j}
    -\varepsilon_b\varepsilon_j\dfrac{\partial g^{\alpha j}}{\partial x^b}\dfrac{\partial g^{\beta j}}{\partial x^a}-\varepsilon_b\varepsilon_j\dfrac{\partial g^{\alpha j}}{\partial x^a}\dfrac{\partial g^{\beta j}}{\partial x^b}\biggr)\biggr].
\end{eqnarray*}
To simplify this let us introduce the following notations
\begin{gather*}
    E^{\alpha\beta}=\varepsilon_b\dfrac{\partial g^{\alpha\beta}}{\partial x^b}u_b,\quad F^{\beta}_a=\dfrac{\partial g^{\beta b}}{\partial x^a}u_b.
\end{gather*}
Then
\begin{eqnarray*}
\gamma_{(3)}^{\alpha\beta ab}u_au_b & = & \dfrac{1}{3}\biggl(2\varepsilon_jE^{\alpha j}E^{\beta j}-2E^{\alpha j}F^{\beta}_{j}-\varepsilon_jF^{\alpha}_{j}F^{\beta}_{j}+E^{\beta j}F^{\alpha}_{j}
\biggr)\\
& = & \dfrac{1}{6}\biggl(\Big(\varepsilon_jF^{\beta}_{j}-E^{\beta j}\Big)\Big(\big(F^{\alpha}_{j}-\varepsilon_jE^{\alpha j}\big)-3\big(\varepsilon_jE^{\alpha j}+F^{\alpha}_{j}\big)\Big)\biggr)\\
& = & \dfrac{2}{3}\varepsilon_j\widetilde{B}^{j\beta}\widetilde{B}^{j\alpha}+2\varepsilon_j\widetilde{B}^{j\beta}\widetilde{C}^{\alpha j}.
\end{eqnarray*}
Thus, we have the form of the matrix $\widetilde{W}^{kj}$
\begin{equation*}
    \widetilde{W}^{kj}=\left(
    \begin{matrix}
        \varepsilon_jI^{jb} & \widetilde{B}^{a\beta} \\
        \widetilde{C}^{\alpha b} & \dfrac{1}{3}\varepsilon_j\widetilde{B}^{j\beta}\widetilde{B}^{j\alpha}+\varepsilon_j\widetilde{B}^{j\beta}\widetilde{C}^{\alpha j}
    \end{matrix}\right),
\end{equation*}
therefore,
\begin{equation*}
    \left(
    \begin{matrix}
        I^{ab} & 0 \\
        -\widetilde{C}^{\alpha b} & \varepsilon_jI^{j\beta}
    \end{matrix}\right)\widetilde{W}^{kj}=\left(
    \begin{matrix}
        \varepsilon_jI^{jb} & \widetilde{B}^{a\beta} \\
        0 & \dfrac{1}{3}\widetilde{B}^{j\beta}\widetilde{B}^{j\alpha}
    \end{matrix}\right),
\end{equation*}
from which we obtain
$$|\det{\widetilde{W}}|=|\det\dfrac{1}{3}\widetilde{B}^{j\beta}\widetilde{B}^{j\alpha}|.$$ From here we have the homogeneity of $\det \widetilde{W}(u)$ of degree $2(n-m)$ in $u$, since the matrix $B$ is represented by the mapping $\Gamma(u,\cdot)\colon S^{\bot}\to S$, where $S^{\bot}$ is $(n-m)$-dimensional.
\begin{lemma}\label{lem:5.2}
    Let us assume 2-step bracket generating hypothesis for the ss-manifold $M$,
    and let $u\in T^*M$. Then for every $u$ with $\langle gu,u\rangle \neq0$ there exists
    $\delta>0$ such that $$|\det{\widetilde{M}}(u)|\geqslant \delta |\langle gu,u\rangle|^{(n-m)}.$$
\end{lemma}
\begin{proof} By Theorem \ref{th:2.4} the mapping $\Gamma(u,\cdot)\colon
S^{\bot}\to S$ is injective for every nonzero $u$ with $\langle
gu,u\rangle\neq0$. From the other hand, $\widetilde{B}^{j\alpha}$
is a matrix for $-\Gamma(u,\cdot)$ by \eqref{eq:ABCD} and, hence,
the matrix for $\widetilde{B}^{j\beta}\widetilde{B}^{j\alpha}$ is
the matrix for $\Gamma(u,\cdot)^{tr}\cdot\Gamma(u,\cdot)$, which
is invertible by injectivity of $\Gamma(u,\cdot)$. Therefore,
$\det{\widetilde{W}}(u)\neq0$ if $\langle gu,u\rangle \neq0$ and
the statement of the lemma holds due to a homogeneity argument.
\end{proof}
\begin{remark}\label{remark1}
Lemma~\ref{lem:5.2} can be reformulated in the following way:
$\det{\widetilde{W}}(u)\neq0$ if and only if $gu$ is a 2-step
bracket generator.\end{remark}
\begin{theorem}
    $i)$ If $gu$ is a $2$-step bracket generator, then there exists $\delta>0$ such that $\exp_p(tu)$ is a local diffeomorphism for any $0<t<\delta$.\\
    $ii)$ Assuming $2$-step bracket generating hypothesis, there exists $\delta>0$ depending continuously on $p$ such that $\exp_p(u)$ is a local diffeomorphism for $u$ near $p=0$ and $\langle gu, u\rangle\neq0$.
\end{theorem}
\begin{proof}
    The assertion $i)$ follows from Lemma \ref{lem:5.2} and the Remark \ref{remark1}.
    By Lemma~\ref{lem:5.1} $$\det W(u)\geqslant \det \widetilde{W}(u)-C|u|^{2(n-m)+1}$$ for small $u$. Thus, $ii)$ holds for $|u|\leqslant \delta \left(\dfrac{|\langle gu,u\rangle|}{|u|^2}\right)^{(n-m)}$.
\end{proof}

\section{Quaternion ss-manifold}\label{sec:6}

In the present chapter we find the parametric equations of
extremals for a group furnished with the sub-semi-Riemannian
metric of the index 2 described earlier in Example 2 at Section 2.
The Hamiltonian function $H(\xi,\theta,x,z)$ has the following
form
\begin{gather}
   H=\frac{1}{2}(-\xi_1^2-\xi_2^2+\xi_3^2+\xi_4^2)+\frac{1}{2}(x_2x_4\theta_1\theta_2+(x_2x_3+x_1x_4)\theta_1\theta_3-x_1x_3\theta_1\theta_2)\notag\\
   +\dfrac{1}{8}(\theta_1^2-\theta_2^2-\theta_3^2)(-x_1^2-x_2^2+x_3^2+x_4^2)\notag\\
   +\frac{1}{2}\theta_1(-x_2\xi_1+x_1\xi_2+x_4\xi_3-x_3\xi_4)+\frac{1}{2}\theta_2(x_4\xi_1+x_3\xi_2+x_2\xi_3+x_1\xi_4)\\ +\frac{1}{2}\theta_3(x_3\xi_1-x_4\xi_2+x_1\xi_3-x_2\xi_4).\notag
\end{gather}
The corresponding Hamiltonian system is
\begin{equation}\label{hamsyst}
\begin{array}{l}
\vspace{1mm}
  \dot{x}_1=\dfrac{\partial H}{\partial \xi_1}=-\xi_1-\frac{1}{2}x_2\theta_1+\frac{1}{2}x_4\theta_2+\frac{1}{2}x_3\theta_3, \\ \vspace{1mm}
  \dot{x}_2=\dfrac{\partial H}{\partial \xi_2}=-\xi_2+\frac{1}{2}x_1\theta_1+\frac{1}{2}x_3\theta_2-\frac{1}{2}x_4\theta_3, \\ \vspace{1mm}
  \dot{x}_3=\dfrac{\partial H}{\partial \xi_3}=\xi_3+\frac{1}{2}x_4\theta_1+\frac{1}{2}x_2\theta_2+\frac{1}{2}x_1\theta_3, \\ \vspace{1mm}
  \dot{x}_4=\dfrac{\partial H}{\partial \xi_4}=\xi_4-\frac{1}{2}x_3\theta_1+\frac{1}{2}x_1\theta_2-\frac{1}{2}x_2\theta_3,
\end{array}
\end{equation}
\begin{equation*}\begin{array}{l}
  \dot{z}_1=\dfrac{\partial H}{\partial \theta_1}=\frac{1}{2}((x_2x_4-x_1x_3)\theta_2+(x_2x_3+x_1x_4)\theta_3)+\frac{1}{4}\theta_1(-x_1^2-x_2^2+x_3^2+x_4^2)\\ \vspace{1mm}
  +\frac{1}{2}(-x_2\xi_1+x_1\xi_2+x_4\xi_3-x_3\xi_4), \\ \vspace{1mm}
  \dot{z}_2=\dfrac{\partial H}{\partial \theta_2}=\frac{1}{2}(x_2x_4-x_1x_3)\theta_1-\frac{1}{4}\theta_2(-x_1^2-x_2^2+x_3^2+x_4^2)
  +\frac{1}{2}(x_4\xi_1+x_3\xi_2+x_2\xi_3+x_1\xi_4), \\ \vspace{1mm}
  \dot{z}_3=\dfrac{\partial H}{\partial \theta_3}=\frac{1}{2}(x_2x_3+x_1x_4)\theta_1-\frac{1}{4}\theta_3(-x_1^2-x_2^2+x_3^2+x_4^2)
  +\frac{1}{2}(x_3\xi_1-x_4\xi_2+x_1\xi_3-x_2\xi_4),
\end{array}
\end{equation*}
\begin{equation}
\begin{array}{l}
  \dot{\xi}_1=-\dfrac{\partial H}{\partial x_1}
   =-\frac12(-x_3\theta_1\theta_2+x_4\theta_1\theta_3)+\frac{1}{4}x_1(\theta_1^2-\theta_2^2-\theta_3^2)-\frac{1}{2}\xi_2\theta_1-
  \frac{1}{2}\xi_4\theta_2-\frac{1}{2}\xi_3\theta_3, \\ \vspace{1mm}
  \dot{\xi}_2=-\dfrac{\partial H}{\partial x_2} =-\frac12(x_4\theta_1\theta_2+x_3\theta_1\theta_3)+\frac{1}{4}x_2(\theta_1^2-\theta_2^2-\theta_3^2)+\frac{1}{2}\xi_1\theta_1-
  \frac{1}{2}\xi_3\theta_2+\frac{1}{2}\xi_4\theta_3, \\ \vspace{1mm}
  \dot{\xi}_3=-\dfrac{\partial H}{\partial x_3} =-\frac12(x_2\theta_1\theta_3-x_1\theta_1\theta_2)-\frac{1}{4}x_3(\theta_1^2-\theta_2^2-\theta_3^2)+\frac{1}{2}\xi_4\theta_1-
  \frac{1}{2}\xi_2\theta_2-\frac{1}{2}\xi_1\theta_3,  \\ \vspace{1mm}
   \dot{\xi}_4=-\dfrac{\partial H}{\partial x_4}  =-\frac12(x_2\theta_1\theta_2+x_1\theta_1\theta_3)-\frac{1}{4}x_4(\theta_1^2-\theta_2^2-\theta_3^2)-\frac{1}{2}\xi_3\theta_1
   -\frac{1}{2}\xi_1\theta_2+\frac{1}{2}\xi_2\theta_3, \\ \vspace{1mm}
   \dot{\theta}_1=-\dfrac{\partial H}{\partial z_1}=0, \\ \vspace{2mm}
   \dot{\theta}_2=-\dfrac{\partial H}{\partial z_2}=0, \\ \vspace{1mm}
   \dot{\theta}_3=-\dfrac{\partial H}{\partial z_3}=0. \notag
\end{array}
\end{equation}

We observe that $\theta_1,\theta_2,\theta_3$ are constants.
Let us remind that the projection of a solution of the Hamiltonian
system onto $(x,z)$-space is called extremal. In order to find it
we will reduce the Hamiltonian system  to the system containing
only $(x_1,x_2,x_3,x_4,z_1,z_2,z_3)$ coordinates. If we express
$\xi_1,\ldots,\xi_4$ from the first 4 equations and substitute them
in the equations of the Hamiltonian system, then we obtain
\begin{align}\notag
     &\dot{\xi}_1=\dfrac{1}{2}(\dot{x}_2\theta_1-\dot{x}_4\theta_2-\dot{x}_3\theta_3),\\ \vspace{1mm}\notag
     &\dot{\xi}_2=\dfrac{1}{2}(-\dot{x}_1\theta_1-\dot{x}_3\theta_2+\dot{x}_4\theta_3),\\ \vspace{1mm}\notag
     &\dot{\xi}_3=\dfrac{1}{2}(\dot{x}_4\theta_1+\dot{x}_2\theta_2+\dot{x}_1\theta_3),\\ \vspace{1mm}\notag
     &\dot{\xi}_4=\dfrac{1}{2}(-\dot{x}_3\theta_1+\dot{x}_1\theta_2-\dot{x}_2\theta_3).
\end{align}
Differentiating first 4 equations and substituting $\dot{\xi}_1,\ldots,\dot{\xi}_4$ there, we get
\begin{align*}
     &\ddot{x}_1=-\dot{x}_2\theta_1+\dot{x}_4\theta_2+\dot{x}_3\theta_3,\\
     &\ddot{x}_2=\dot{x}_1\theta_1+\dot{x}_3\theta_2-\dot{x}_4\theta_3,\\
     &\ddot{x}_3=\dot{x}_4\theta_1+\dot{x}_2\theta_2+\dot{x}_1\theta_3,\\
     &\ddot{x}_4=-\dot{x}_3\theta_1+\dot{x}_1\theta_2-\dot{x}_2\theta_3
\end{align*}
or
\begin{equation}\label{sys:1}
    \left(%
   \begin{array}{l}
       \ddot{x}_1\\ \ddot{x}_2\\ \ddot{x}_3\\ \ddot{x}_4
   \end{array}%
    \right)=\left(%
   \begin{array}{cccc}
       0& -\theta_1 & \theta_3 & \theta_2 \\
      \theta_1 & 0 &\theta_2&-\theta_3\\
      \theta_3&\theta_2&0&\theta_1\\
      \theta_2&-\theta_3&-\theta_1&0
   \end{array}%
    \right)\left(%
   \begin{array}{l}
       \dot{x}_1\\ \dot{x}_2\\ \dot{x}_3\\ \dot{x}_4
   \end{array}%
    \right).
\end{equation}
We are looking for the solution $x_1=x_1(t), \ldots,x_4=x_4(t)$, $t\in[-\infty,+\infty]$, satisfying $x_1(0)=0,\ldots,x_4(0)=0$ and $\dot{x}_1(0)=\dot{x}_1^0,\ldots,\dot{x}_4(0)=\dot{x}_4^0$.
The eigenvalues of the matrix
\begin{gather}
    A:=\left(%
   \begin{array}{cccc}
       0& -\theta_1 & \theta_3 & \theta_2 \\
      \theta_1 & 0 &\theta_2&-\theta_3\\
      \theta_3&\theta_2&0&\theta_1\\
      \theta_2&-\theta_3&-\theta_1&0
   \end{array}%
    \right)\notag
\end{gather}
are  $\lambda_1=a$, $\lambda_2=-a$, $\lambda_3=\overline{a}$, and $\lambda_4=-\overline{a}$, where $a=|k|+i\theta_1$, $\overline{a}=|k|-i\theta_1$ and $k=\theta_2+i\theta_3$, $\overline{k}=\theta_2-i\theta_3$. The associated eigenvectors are
\begin{align*}
    &v_1=(ia|k|, a|k|, ak,iak),\\
    &v_2=(-ia|k|,a|k|,-a\overline{k},ia\overline{k}),\\
    &v_3=(i\overline{a}|k|,-\overline{a}|k|,-\overline{a}\overline{k},i\overline{a}\overline{k}),\\
    &v_4=(-i\overline{a}|k|,-\overline{a}|k|,\overline{a}k,i\overline{a}k),
\end{align*} where $|k|=\sqrt{\theta_2^2+\theta_3^2}$.
Notice that the matrix $A$ is skew-symmetric with respect to our nondegenerate metric $Q$ with index 2) in the sense that $Q\cdot(Ax)(y)=-Qx\cdot Ay$. This extends the idea of sub-Riemannian case, which was considered in \cite{Markina}, where the matrix $Q$ was just a unit matrix and $A$ was skew-symmetric in the usual sense. Also it carries on the sub-Lorentzian case, where  $A$ was skew-symmetric with respect to sub-Lorentzian metric $Q$ \cite{KM}.

The solution of the system (\ref{sys:1}) is of the form
\begin{eqnarray*}
\dot{x}_1(t) & = & i|k|(c_1a e^{at}-c_2ae^{-at}+c_3\overline{a}e^{\overline{a}t}-c_4\overline{a}e^{-\overline{a}t}),\notag\\
\dot{x}_2(t) & = & |k|(c_1a e^{at}+c_2ae^{-at}-c_3\overline{a}e^{\overline{a}t}-c_4\overline{a}e^{-\overline{a}t}),\\
\dot{x}_3(t) & = & c_1ak e^{at}-c_2a\overline{k}e^{-at}-c_3\overline{a}\overline{k}e^{\overline{a}t}+c_4\overline{a}ke^{-\overline{a}t},\notag\\
\dot{x}_4(t) & = & i(c_1ak e^{at}+c_2a\overline{k}e^{-at}+c_3\overline{a}\overline{k}e^{\overline{a}t}+c_4\overline{a}ke^{-\overline{a}t}),\nonumber
\end{eqnarray*}
where
\begin{eqnarray}\label{const}
c_1 & = & \dfrac{1}{4iak|k|}\cdot (k(\dot{x}_1^0+i\dot{x}_2^0)+|k|(\dot{x}_4^0+i\dot{x}^0_3)),\notag\\
c_2 & = & \dfrac{1}{4ia\overline{k}|k|}\cdot (-\overline{k}(\dot{x}_1^0-i\dot{x}_2^0)+|k|(\dot{x}_4^0-i\dot{x}^0_3)),\\
c_3 & = & \dfrac{1}{4i\overline{a}\overline{k}|k|}\cdot (\overline{k}(\dot{x}_1^0-i\dot{x}_2^0)+|k|(\dot{x}_4^0-i\dot{x}^0_3)),\notag\\
c_4 & = & \dfrac{1}{4i\overline{a}k|k|}\cdot (-k(\dot{x}_1^0+i\dot{x}_2^0)+|k|(\dot{x}_4^0+i\dot{x}^0_3)).\notag
\end{eqnarray}
Therefore, the $x$-coordinates of the extremals have a form
\begin{align}\label{eq:x}
    &x_1(t)=i|k|(c_1e^{at}+c_2e^{-at}+c_3e^{\overline{a}t}+c_4e^{-\overline{a}t})-i|k|(c_1+c_2+c_3+c_4),\nonumber\\
    &x_2(t)=|k|(c_1e^{at}-c_2e^{-at}-c_3e^{\overline{a}t}+c_4e^{-\overline{a}t})-|k|(c_1-c_2-c_3+c_4),\\
    &x_3(t)=c_1ke^{at}+c_2\overline{k}e^{-at}-c_3\overline{k}e^{\overline{a}t}-c_4ke^{-\overline{a}t}-(c_1k+c_2\overline{k}-c_3\overline{k}-c_4k),\nonumber\\
    &x_4(t)=i(c_1ke^{at}-c_2\overline{k}e^{-at}+c_3\overline{k}e^{\overline{a}t}-c_4ke^{-\overline{a}t})-i(c_1k-c_2\overline{k}+c_3\overline{k}-c_4k). \nonumber
\end{align} From the horizontality conditions
\begin{eqnarray*}
\dot z_1 & = & \frac{1}{2}(+x_2 \dot x_1-x_1\dot x_2+x_4\dot x_3-x_3\dot x_4),\notag\\
\dot z_2 &= & \frac{1}{2}(-x_4 \dot x_1-x_3\dot x_2+x_2\dot x_3+x_1\dot x_4),\\
\dot z_3 & = & \frac{1}{2}(-x_3 \dot x_1+x_4\dot x_2+x_1\dot x_3-x_2\dot x_4)\notag
\end{eqnarray*} we can find the vertical components
\begin{eqnarray}\label{eq:z}
z_1(t) & = & 2i|k|^2(-2(c_1c_2a-c_3c_4\overline{a})t+c_1c_2(e^{at}-e^{-at})-c_3c_4(e^{\overline{a}t}-e^{-\overline{a}t})),\\
z_2(t) & = & 2\theta_2|k|(-2(c_1c_2a+c_3c_4\overline{a})t
+c_1c_2(e^{at}-e^{-at})+c_3c_4(e^{\overline{a}t}-e^{-\overline{a}t}))\nonumber
\\
& + & 2\theta_1\theta_3(c_1c_3e^{2|k|t}+c_2c_4e^{-2|k|t}-c_1c_3-c_2c_4)\nonumber \\
& + & 2i\theta_3|k|(c_1c_3e^{at}+c_2c_4e^{-at}-c_1c_3e^{\overline{a}t}-c_2c_4e^{-\overline{a}t}),\nonumber\\
z_3(t) & = & 2\theta_3|k|(-2(c_1c_2a+c_3c_4\overline{a})t+c_1c_2(e^{at}-e^{-at})
+c_3c_4(e^{\overline{a}t}-e^{-\overline{a}t}))\nonumber\\
& - &
2\theta_1\theta_2(c_1c_3e^{2|k|t}+c_2c_4e^{-2|k|t}-c_1c_3-c_2c_4)\nonumber
\\
& - &
2i\theta_2|k|(c_1c_3e^{at}+c_2c_4e^{-at}-c_1c_3e^{\overline{a}t}-c_2c_4e^{-\overline{a}t}).\nonumber
\end{eqnarray}
The constants of integration $c_i$, $i=1,\ldots,4$ are given
by~\eqref{const} through the initial velocity. We would like to
calculate the homogeneous norm of an element $\big(x(t),z(t)\big)$
given by
$$\|(x,z)\|^4=(-x_1^2-x_2^2+x_3^2+x_4^2)^2+z_1^2+z_2^2+z_3^2.$$ We
have
\begin{eqnarray}\label{nx}\|x(t)\|^2 & = & (-x_1^2-x_2^2+x_3^2+x_4^2)(t)  \\
& = & 8|k|^2(2c_1c_2+2c_3c_4-c_1c_2(e^{at}+e^{-at})-c_3c_4(e^{\overline{a}t}+e^{-\overline{a}t}))\nonumber\\
& = & -32|k|^2(c_1c_2\sinh^2\frac{at}{2}+c_3c_4\sinh^2\frac{\bar
at}{2}).\nonumber\end{eqnarray} Let us introduce the notation
$w_1=\frac{k}{|k|}(\dot{x}_1^0+i\dot{x}_2^0)$,
$w_2=\dot{x}_4^0+i\dot{x}^0_3$ in order to simplify the
calculations. Then
\begin{gather*}
     c_1c_2=-\dfrac{1}{16a^2|k|^2}\,\big(|w_2|^2-|w_1|^2+2i\im(w_1\bar w_2)\big)\end{gather*}
and
\begin{gather*}c_3c_4=-\dfrac{1}{16\overline{a}^2|k|^2}\,\big(|w_2|^2-|w_1|^2-2i\im(w_1\bar w_2)\big).
\end{gather*} We see that $c_1c_2=\overline{c_3c_4}$ and~\eqref{nx} takes the form $$\|x(t)\|^2=-64|k|^2\re(c_1c_2\sinh^2\frac{at}{2}),\quad a=|k|+i\theta_1.$$
We also need the values $$c_1c_3=-\dfrac{1}{16|a|^2|k|^2}\,|w_2+w_1|^2,$$ $$c_2c_4=-\dfrac{1}{16|a|^2|k|^2}\,|w_2-w_1|^2,$$
and $$c_1c_2c_3c_4=-\dfrac{1}{16^2|a|^4|k|^4}\,|w_2^2-w_1^2|^2$$
Then
\begin{gather*}
    (z_1^2+z_2^2+z_3^2)(t)=4|k|^4\Big(16t^2c_1c_2c_3c_4a\overline{a}-8c_1c_2c_3c_4\overline{a}t(e^{at}-e^{-at})-8c_1c_2c_3c_4at(e^{\overline{a}t}-e^{-\overline{a}t})\\+4c_1c_2c_3c_4(e^{at}-e^{-at})(e^{\overline{a}t}-e^{-\overline{a}t})+c_1^2c_3^2(2e^{2|k|t}-e^{2at}-e^{2\overline{a}t})+c_2^2c_4^2(2e^{-2|k|t}-e^{-2at}-e^{-2\overline{a}t})\\-4c_1c_2c_3c_4+2c_1c_2c_3c_4(e^{2i\theta_1t}+e^{-2i\theta_1t})\Big)\\
    +4\theta_1^2|k|^2\Big(c_1c_2c_3c_4(4-2(e^{2|k|t}+e^{-2|k|t}))-2(c_1c_3^2e^{2|k|t}+c_2^2c_4^2e^{-2|k|t})\\+(c_1^2c_3^2e^{4|k|t} + c_2^2c_4^2e^{-4|k|t})
    +(c_1^2c_3^2+c_2^2c_4^2)\Big)\\
   + 8i\theta_1|k|^3\Big(c_1^2c_3^2(e^{(a+2|k|)t}-e^{(\overline{a}+2|k|)t}-e^{at}+e^{\overline{a}t})+c_2^2c_4^2(e^{-(a+2|k|)t}-e^{-(\overline{a}+2|k|)t}-e^{-at}+e^{-\overline{a}t})\\+c_1c_2c_3c_4((e^{(a-2|k|)t}+e^{-(a-2|k|)t})-(e^{(\overline{a}-2|k|)t}+e^{-(\overline{a}-2|k|)t})-(e^{at}+e^{-at})+(e^{\overline{a}t}+e^{-\overline{a}t}))\Big).
\end{gather*}

To simplify the last expression we notice that
\begin{gather*}
    16|k|^4c_1c_2c_3c_4\Big(16t^2a\overline{a}-8t(\overline{a}e^{at}-\overline{a}e^{-at}+ae^{\overline{a}t}-ae^{-at})+(e^{2|k|t}+e^{-2|k|t}-e^{2i\theta_1t}-e^{-2i\theta_1t})\Big)\\
    =64c_1c_2c_3c_4|k|^4(t^2(|k|^2+\theta_1^2)-2t(|k|\sinh(|k|t)\cos(\theta_1t)+\theta_1\cosh(|k|t)\sin(\theta_1t))\\+\sinh^2(|k|t)+\sin^2(\theta_1t))\\
    =64c_1c_2c_3c_4|k|^4\Big((|k|t-\sinh(|k|t)\cos(\theta_1t))^2+(\theta_1t-\cosh(|k|t)\sin(\theta_1t))^2\Big),
\end{gather*}
and
\begin{gather*}
    4|k|^4\Big(c_1^2c_3^2(2e^{2|k|t}-e^{2at}-e^{2\overline{a}t})+c_2^2c_4^2(2e^{-2|k|t}-e^{-2at}-e^{-2\overline{a}t})\\-4c_1c_2c_3c_4+2c_1c_2c_3c_4(e^{2i\theta_1t}+e^{-2i\theta_1t})\Big)\\
    = 4|k|^4\Big(4\sin^2(\theta_1 t)(c_1^2c_3^2e^{2|k|t}+c_2^2c_4^2e^{-2|k|t})-8c_1c_2c_3c_4\sin^2(\theta_1t)\Big)\\
    =16|k|^4\sin^2(\theta_1 t)(c_1c_3e^{|k|t}-c_2c_4e^{-|k|t})^2,
\end{gather*}
and
\begin{gather*}
 4\theta_1^2|k|^2\Big(c_1c_2c_3c_4(4-2(e^{2|k|t}+e^{-2|k|t}))-2(c_1c_3^2e^{2|k|t}+c_2^2c_4^2e^{-2|k|t})\\+(c_1^2c_3^2e^{4|k|t}
 +  c_2^2c_4^2e^{-4|k|t})
 +(c_1^2c_3^2+c_2^2c_4^2)\Big)\\
 =  4\theta_1^2|k|^2\Big(-8c_1c_2c_3c_4\sinh^2|k|t+c_1^2c_3^2(1-e^{2|k|t})^2+c_2^2c_4^2(1-e^{-2|k|t})^2\Big)\\
 =  16\theta_1^2|k|^2\sinh^2(|k|t)\Big(c_1c_3e^{|k|t}-c_2c_4e^{-|k|t}\Big)^2,
\end{gather*}
and
\begin{gather*}
    8i\theta_1|k|^3\Big(c_1^2c_3^2(e^{(a+2|k|)t}-e^{(\overline{a}+2|k|)t}-e^{at}+e^{\overline{a}t})+c_2^2c_4^2(e^{-(a+2|k|)t}-e^{-(\overline{a}+2|k|)t}-e^{-at}+e^{-\overline{a}t})\\+c_1c_2c_3c_4((e^{(a-2|k|)t}+e^{-(a-2|k|)t})-(e^{(\overline{a}-2|k|)t}+e^{-(\overline{a}-2|k|)t})-(e^{at}+e^{-at})+(e^{\overline{a}t}+e^{-\overline{a}t}))\Big)\\
    =8i\theta_1|k|^3\Big(4i\sinh(|k|t)\sin(\theta_1t)(c_1^2c_3^2e^{2|k|t}+c_2^2c_4^2e^{-2|k|t})-8ic_1c_2c_3c_4\sinh(|k|t)\sin(\theta_1t)\Big)\\
    =-32\theta_1|k|^3\sinh(|k|t)\sin(\theta_1t)(c_1c_3e^{|k|t}-c_2c_4e^{-|k|t})^2.
\end{gather*}
Then joining last expressions we get
\begin{gather*}
    (z_1^2+z_2^2+z_3^2)(t)=16|k|^2(c_1c_3e^{|k|t}-c_2c_4e^{-|k|t})^2(\theta_1\sinh(|k|t)-|k|\sin(\theta_1t))^2\\+64|k|^4c_1c_2c_3c_4\left((|k|t-\sinh(|k|t)\cos(\theta_1t))^2+(\theta_1t-\cosh(|k|t)\sin(\theta_1t))^2\right).
\end{gather*}
The letter calculations we can summarize in the following
\begin{theorem}
Let $Q$ be ss-manifold described in Example~2 of Section~2. The
normal extremals starting from the identity of the group $Q$ are
given by the solution of the Hamiltonian system~\eqref{hamsyst} by
the parametric equations~\eqref{eq:x} and~\eqref{eq:z}. The
constant of integrations $c_k$, $k=1,2,3,4$, are related to the
initial velocity $(\dot x_1^0,\dot x_2^0,\dot x_3^0,\dot x_4^0)$
by~\eqref{const}, where $k=\theta_2+i\theta_3$ and
$a=|k|+i\theta_1$ are the first integrals of the Hamiltonian
system~\eqref{hamsyst}.
\end{theorem}

\end{document}